\documentclass{amsart}

\usepackage{amsfonts,amssymb,amsthm}
\usepackage{mathrsfs}
\usepackage{enumitem}
\usepackage{hyperref}
\usepackage{graphicx}
\usepackage{subfigure}
\usepackage{psfrag}
\usepackage{xfrac}
\usepackage{color}
\usepackage{overpic}

\newtheorem{theorem}{Theorem}[section]
\newtheorem{proposition}[theorem]{Proposition}
\newtheorem{corollary}[theorem]{Corollary}
\newtheorem{lemma}[theorem]{Lemma}
{\theoremstyle{definition}
\newtheorem{definition}[theorem]{Definition}

\newtheorem{example}[theorem]{Example}

}

\newcommand{\Z}{\mathbb{Z}}
\newcommand{\R}{\mathbb{R}}
\newcommand{\C}{\mathbb{C}}

\begin{document}

\title[Topological equivalence of foliations and functions]{Topological equivalence of submersion functions and topological equivalence of their foliations on the plane: the linear-like case}

\author[F. Braun \MakeLowercase{and} I.S. Meza-Sarmiento]
{Francisco Braun \MakeLowercase{and} Ingrid S. Meza-Sarmiento}

\address{Departamento de Matem\'{a}tica, Universidade Federal de S\~ao Carlos, 
13565--905 S\~ao Carlos, S\~ao Paulo, Brazil}
\email{franciscobraun@ufscar.br}

\email{isofia1015@gmail.com}

\subjclass[2010]{Primary: 37C10, 57R30, 58K65; Secondary: 34C99.}

\keywords{Regular foliation, topological classification, submersions}

\thanks{This paper is in final form and no version of it will be submitted for publication elsewhere.}

\date{\today}

\maketitle

\begin{abstract}
Let $f, g: \R^2 \to \R$ be two submersion functions and $\mathscr{F}(f)$ and $\mathscr{F}(g)$ be the regular foliations of $\R^2$ whose leaves are the connected components of the levels sets of $f$ and $g$, respectively. 
The topological equivalence of $f$ and $g$ implies the topological equivalence of $\mathscr{F}(f)$ and $\mathscr{F}(g)$, but the converse is not true, in general. 
In this paper, we introduce the class of linear-like submersion functions, which is wide enough in order to contain non-trivial behaviors, and provide conditions for the validity of the converse implication for functions inside this class. 
Our results lead us to a complete topological invariant for topological equivalence in a certain subclass of linear-like submersion functions. 
\end{abstract}

\section{Introduction} 
We say that two functions $f, g: \R^2 \to \R$ are \emph{topologically equivalent} (resp. o-\emph{topologically equivalent}) if there exist homeomorphisms (resp. orientation preserving homeomorphisms) $h: \R^2 \to \R^2$ and $\ell :\R\to\R$ such that 
$$
\ell \circ f = g \circ h. 
$$ 
On the other hand,  if $\mathscr{F}$ and $\mathscr{G}$ are regular foliations of open subsets $A$ and $B$ of $\R^2$, respectively, we say they are \emph{topologically equivalent} (resp. o-\emph{topologically equivalent}) if there exists a homeomorphism (resp. orientation preserving homeomorphisms) $h:A\to B$ carrying each leaf of $\mathscr{F}$ onto a leaf of $\mathscr{G}$. 
In this case we call $h$ an \emph{equivalence homeomorphism} (resp. \emph{o-equivalence homeomorphism}) between $\mathscr{F}$ and $\mathscr{G}$. 

Let $f: \R^2 \to \R$ be a $C^{\infty}$ submersion function. 
It is well known that $f$ induces a regular foliation of $\R^2$ whose leaves are the connected components of $f$. 
We denote this foliation by $\mathscr{F}(f)$.

If $f$ and $g$ are topologically equivalent $C^{\infty}$ submersion functions of $\R^2$, the homeomorphism $h$ of the definition carries level sets of $f$ onto level sets of $g$, preserving the connected components. 
Hence $h$ is an equivalence homeomorphism proving that $\mathscr{F}(f)$ and $\mathscr{F}(g)$ are topologically equivalent. 
Clearly this implication is also true for o-topologically equivalences. 
But the converse is not true, even in the polynomial case. 
Indeed, take 
\begin{equation}\label{ex122}
\begin{aligned}
p(x,y) & = x (7 x - 5)  + (x+1)^2 x^2 (x-2)^2 y, \\ 
q(x,y) & = 2 x (4 x - 5) + (x+1)^2 x^2 (x-2)^2 y. 
\end{aligned}
\end{equation}
In Example \ref{degree2} below we prove that $\mathscr{F}(p)$ and $\mathscr{F}(q)$ are o-topologically equivalent although $p$ and $q$ are not topologically equivalent. 
The main objective of this paper is to study this converse, i.e. to provide conditions in order that the topological equivalence of $\mathscr{F}(p)$ and $\mathscr{F}(q)$ implies the topological equivalence of $p$ and $q$, for two given submersion functions $p, q: \R^2\to \R$. 
At least for a class of submersions to be defined soon. 

The question of characterizing the topological equivalence of functions has long been studied, in different frameworks. 
In the theory of singularities the local topological equivalence of smooth or analytic functions is a classical subject since the works of M. Morse \cite{morse-46, morse-47} in the topological classification of continuous functions on $2$-dimensional manifolds. 
In the study of integrable systems on the plane, it is possible to relate the topological equivalence of such systems to the topological equivalence of their first integrals. 
For instance, J. Mart\'inez-Alfaro et al. \cite{mmo-13,mmo-15} gave complete topological invariants for the integrable Morse-Bott systems and for their first integrals, the so called Morse-Bott functions. 
In a similar line, K.R. Meyer \cite{Me} associated a function, called $\xi$-function, to each Morse-Smale system (a $\xi$-functions is not a first integrals of the system, as its level sets may be transversal to it), and provided conditions in order to obtain the equivalence between topological equivalence of two Morse-Smale systems and the topological equivalence of the related $\xi$-functions. 

Although the topological equivalence of functions is a very studied subject and there are local topological invariants for its study, global results and complete invariants are not so common in the literature. 

Without singularities and so more close to our subject, V.V. Sharko and Y.Y. Soroka \cite{sharko-15} proved that a continuous function $f:\R^2\to\R$, without critical points, is topologically equivalent to a projection if and only if the level sets of $f$ are connected and non-empty. 
In the $C^{\infty}$ higher dimensional case, by using different tools, M. Tib\u ar and A. Zaharia \cite[Prop. 2.7]{TZ} proved that if $M\subseteq \R^n$ is a $C^{\infty}$ submanifold of dimension $m+1$ and $g:M\to \R^m$ is a $C^{\infty}$ function without critical values having level sets diffeomorphic to $\R$ and closed in $\mathbb R^n$, then $g$ is a $C^{\infty}$ \emph{trivial fibration} (see the definition further in this work), in other words $g$ is topologically equivalent to a projection. 
Moreover, $M$ is diffeomorphic to $\R^{m+1}$. 

The foliation $\mathscr{F}(f)$ is \emph{trivial}, i.e. without separatrices (see the precise definition further in this paper), in Sharko and Soroka's case, in fact the very same result says that connectedness of level sets of $f$ is equivalent to triviality of $\mathscr{F}(f)$. 
So, in particular, when two surjective submersion functions have trivial foliation they are topologically equivalent. 

The level sets of a submersion $f: \R^2 \to \R$ need not be connected or, equivalently, the foliation $\mathscr{F}(f)$ need not be trivial, as the examples in \eqref{ex122} above illustrate, and the right above-mentioned equivalence does not hold. 
So what are the additional hypothesis we have to assume in order to guarantee that two given submersion functions $f,g: \R^2 \to \R$ having topologically equivalent foliations $\mathscr{F}(f)$ and $\mathscr{F}(g)$ are topologically equivalent?  

In this paper we address this question for the class of submersions of the form 
\begin{equation}\label{linearform}
p(x,y) = r(x) + s(x) y, 
\end{equation}
where $r(x)$ and $s(x)$ are $C^{\infty}$ real functions. 
We say that $p$ as in \eqref{linearform} is a \emph{linear-like} function. 
This class of functions is wide enough for containing trivial submersions (when $s(x) \neq 0$ for all $x$) and more interesting examples, as the known S.A. Broughton's function \cite{B} $x + x^2 y$ and yet more complicated ones, for instance our examples above as well as examples presenting limit-separatrices in the associated foliation, see sections \ref{section:linear} and \ref{section:examples}. 
We answer the above question for the subclass of \emph{finite linear-like submersions}, defined to be the linear-like ones where the zero set of $s$ is finite, in the results below. 
In order to state them, we denote by $B(p)$ the bifurcation set of a function $p$, see Definition \ref{definition:tf} below. 
\begin{theorem}\label{theorem:main}
Two finite linear-like submersions $p,\,q:\R^2\to\R$ are o-topologically equivalent if and only if $B(p) = B(q) = \emptyset$ or the following conditions hold:
\begin{enumerate}[label={\textnormal{(\alph*)}}]
\item\label{c1} $\mathscr{F}(p)$ and $\mathscr{F}(q)$ are o-topologically equivalent. 
\item\label{c2} There exists an increasing bijection $\sigma:B(p)\to B(q)$ such that 
$$
p^{-1}(\sigma^{-1}(c)) = h^{-1}(q^{-1}(c)) 
$$ 
for each $c\in B(q)$, where $h: \R^2 \to \R^2$ is an o-equivalence homeomorphism between $\mathscr{F}(p)$ and $\mathscr{F}(q)$. 
If $B(p)$ is a singleton, there further exist $m \notin B(p)$ and an extension of $\sigma$ to an increasing bijection $\overline{\sigma}: B(p) \cup \{m\} \to B(q) \cup \{\overline{\sigma}(m)\}$ such that $p^{-1}(m) = h^{-1}(q^{-1}(\overline{\sigma}(m)))$. 
\end{enumerate}
\end{theorem}

The appearing of the bifurcation set in Theorem \ref{theorem:main} becomes natural if one is aware that A. Bodin and M. Tib\u ar \cite{BT} proved that two \emph{complex} polynomials $p$ and $q$ with $n$ variables, $n\neq 3$, are topologically equivalent if and only if it is possible to deform $p$ into $q$ by means of a continuous family of polynomials $h_s$, with $h_0=p$ and $h_1=q$, such that $\deg h_s$, the Euler characteristic of the general fibre of $h_s$ and the cardinality of $B(h_s)$ are independent of $s$. 
Our approach is different as we do not use deformations, instead we use the foliations associated to $p$ and $q$ and the relation in statement \ref{c2}. 

Anyway, even in the polynomial linear-like case, only assuming that the general fibers of $p$ and $q$ have the same Euler characteristic and the same bifurcation sets is not sufficient for the topological equivalence of $p$ and $q$: in Example \ref{degree2} we will explain that the polynomials of degree $7$ $p$ and $q$ of \eqref{ex122} have the same bifurcation set and generic fibers with the same Euler characteristic (with $\mathscr{F}(p)$ and $\mathscr{F}(q)$ topologically equivalent), but even so $p$ and $q$ are not topologically equivalent. 

Condition \ref{c2} is equivalent to ask the relation of topologically equivalence of $p$ and $q$, $\ell \circ p = q \circ h$, in the pre-image of $B(p)$ (or $B(p)$ plus a point) by $p$. 
Actually we just need to test this condition in a smaller set of points, according to the more general Theorem \ref{mmain} below, see also Theorem \ref{corollary:main}. 

When $B(p)$ and $B(q)$ are singletons, we need to add points to them in condition \ref{c2} in order to guarantee o-topological equivalence: indeed $p$ and $-p$, where $p(x,y) = x + x^3 y$, are not o-topologically equivalent although $p$ and $-p$ are topologically equivalent and $\mathscr{F}(p)$ and $\mathscr{F}(-p)$ are o-topologically equivalent, see Example \ref{dd45} below. 
If we are just concerned with topological equivalence, then we can rephrase Theorem \ref{theorem:main} as 
\begin{corollary}\label{coro:main}
Two finite linear-like submersions $p,\,q:\R^2\to\R$ are topologically equivalent if and only if $B(p) = B(q) = \emptyset$ or the following conditions hold:
\begin{enumerate}[label={\textnormal{(\alph*)}}]
\item $\mathscr{F}(p)$ and $\mathscr{F}(q)$ are topologically equivalent. 
\item There exists a monotone bijection $\sigma:B(p)\to B(q)$ such that 
$
p^{-1}(\sigma^{-1}(c)) = h^{-1}(q^{-1}(c))
$ 
for each $c\in B(q)$, where $h: \R^2 \to \R^2$ is an equivalence homeomorphism between $\mathscr{F}(p)$ and $\mathscr{F}(q)$. 
\end{enumerate}
\end{corollary}

See also the more general Corollary \ref{cor5.1} below. 
In condition (b), we can not assure that $\sigma$ is increasing if $p$ and $q$ are only topologically equivalent, see Example \ref{only}. 

From our results we obtain \emph{complete topological invariants} for topological and o-topological equivalence of finite linear-like submersions by putting together the \emph{separatrix configuration}  $S\mathfrak{S}_{\mathscr{F}(p)}$ of the foliation $\mathscr{F}(p)$, according to Definition \ref{sepcon} below, and some ``order'' in the image of the set of \emph{separatrices} inside the bifurcation set $B(p)$ of $p$ (for precise definitions of these concepts, see Section \ref{section:preliminaries}): 
\begin{theorem}\label{corollary:main}
Two finite linear-like submersions $p, q: \R^2 \to \R$ are o-topologically equivalent if and only if $B(p) = B(q) = \emptyset$ or the following conditions hold: 
\begin{enumerate}[label={\textnormal{(\alph*)}}]
\item The separatrix configurations of $\mathscr{F}(p)$ and $\mathscr{F}(q)$ are isomorphic. 
\item There exists an increasing bijection $\sigma: B(p) \to B(q)$ such that $\sigma\circ p(\gamma) = q\circ g(\gamma)$ for every separatrix $\gamma \in \mathfrak{S}_{\mathscr{F}(p)}$, where $g$ is an isomorphism between $S\mathfrak{S}_{\mathscr{F}(p)}$ and $S\mathfrak{S}_{\mathscr{F}(q)}$. 
If $B(p)$ is a singleton, there further exist an ordinary leaf $\xi$ in $S\mathfrak{S}_{\mathscr{F}(p)}$ with $p(\xi) = m \notin B(p)$ and an extension of $\sigma$ to an increasing bijection $\overline{\sigma}: B(p) \cup\{m\} \to B(q) \cup \{\overline{\sigma}(m)\}$ such that $\overline{\sigma}(m) = q(g(\xi))$.  
\end{enumerate}
\end{theorem} 
Since there are not vanishing phenomenon in the linear-like submersion case, according to Corollary \ref{12091} below, the image of the separatrices of $\mathscr{F}(p)$ agrees with the set $B(p)$, so the above relations make sense. 
For topological equivalence we have 
\begin{corollary}\label{varcorollary:main}
Two finite linear-like submersions $p, q: \R^2 \to \R$ are topologically equivalent if and only if $B(p) = B(q) = \emptyset$ or the following conditions hold: 
\begin{enumerate}[label={\textnormal{(\alph*)}}]
\item The separatrix configurations of $\mathscr{F}(p)$ and $\mathscr{F}(q)$ are isomorphic or anti-isomorphic. 
\item There exists a monotone bijection $\sigma:B(p)\to B(q)$ such that $\sigma\circ p(\gamma) = q\circ g(\gamma)$ for every $\gamma \in \mathfrak{S}_{\mathscr{F}(p)}$, where $g$ is an isomorphism or anti-isomorphism between $S\mathfrak{S}_{\mathscr{F}(p)}$ and $S\mathfrak{S}_{\mathscr{F}(q)}$. 
\end{enumerate}
\end{corollary}

The organization of the paper is as follows: 
Section \ref{section:preliminaries} is mainly devoted to recall some results on regular foliations on the plane, as well as bifurcation theory. 
In Section \ref{section:linear} we study the linear-like functions. 
In Section \ref{section:examples} we give examples of linear-like submersions with different kind of behaviors. 
The main results stated above are then proved in Section \ref{section:final} as a consequence of the more general Theorem \ref{mmain}.

In order to simplify notations, sometimes throughout the text a singleton will be viewed as just its element. 

\section{Regular foliations and bifurcation sets}\label{section:preliminaries}

\subsection{Regular foliations, chordal relations and a complete topological invariant}\label{890jd}
Given an open set $U\subset \R^2$, we say that a collection of curves $\mathscr{F}$ contained in $U$ is a \emph{regular foliation} of $U$ if 
$$
U = \bigcup_{\gamma\in \mathscr{F}} \gamma,
$$ 
and for any $P\in U$ there exists an open neighborhood of $P$, $N(P)$, and a homeomorphism $h: N(P) \to \R^2$ such that each curve of $\mathscr{F}|_{N(P)} \dot{=} \{\gamma\cap N(P)\ |\  \gamma\in \mathscr{F}\}$ is carried by $h$ onto a line $x=\textnormal{const}$. 
A curve of $\mathscr{F}$ will be also called \emph{leaf}. 

For instance, let $f:\mathbb{R}^2\to\mathbb R$ be a $C^{\infty}$ submersion function. 
From the implicit function Theorem, $f$ defines a regular foliation of $\R^2$, denoted by $\mathscr{F}(f)$, whose leaves are the connected components of the level sets of $f$ 
$$ 
L_c=L_c(f)= f^{-1}(c), 
$$  
$c\in f(\mathbb R^2)$. 
Foliations $\mathscr{F}(f)$ of $\R^2$ will play the main role of this paper. 

To any foliation $\mathscr{F}$ of an open set $U\subset \R^2$ we associate the \emph{space of leaves} $U/\mathscr{F}$, defined to be the quotient space $U/{\mathord\sim}$, where $\sim$ is the equivalence relation in $U$ saying that two points are related if they belong to the same leaf of $\mathscr{F}$. 
This topological space, with the quotient topology, may not be Hausdorff. 

Two leaves of $\mathscr{F}$ whose classes can not be separated in $U/\mathscr{F}$ are said \emph{inseparable leaves}. 
We can equivalently define them as follows: two leaves $l_1, l_2 \in \mathscr{F}$ are inseparable if for any cross sections $C_1$ and $C_2$ through $l_1$ and $l_2$ there exists another leaf $l_3\in \mathscr{F}$ intersecting $C_1$ and $C_3$. 

We say that a leaf $l\in \mathscr{F}$ is a \emph{separatrix} if it lives in the closure of the set of the inseparable leaves. 
A separatrix which is not an inseparable leaf is called a \emph{limit separatrix}. 
The set of separatrices of $\mathscr{F}$ will be denoted by $\mathfrak{S}_{\mathscr{F}}$. 
This is a closed set by definition. 
Each connected component of the complement of $\mathfrak{S}_{\mathscr{F}}$ will be called a \emph{canonical region}. 

From now on we consider regular foliations of the entire plane. 
Each curve $\gamma$ of such a foliation tends to infinity in both directions and so divides the plane into two open regions such that $\gamma$ is the common boundary. 
Rephrasing W. Kaplan \cite{K,K1}, we can view the leaves of such a foliation as ``chords'' laying on a disc. 
Kaplan managed to prove that the relations between triples of leaves of the foliation, the \emph{chordal relations}, characterize it as we will recall in Theorem \ref{KK}. 
First we recall the chordal relations. 

Given $l_1$, $l_2$ and $l_3$ three different leaves of $\mathscr{F}$, we say that $l_2$ \emph{separates} $l_1$ and $l_3$, denoting this by $l_1|l_2|l_3$ or $l_3|l_2|l_1$, if $l_1$ and $l_3$ are in different connected components of $\R^2\setminus l_2$. 
If $l_i$ does not separate $l_j$ and $l_k$ for any different $i,j,k \in \{1,2,3\}$ then we say that $l_1$, $l_2$ and $l_3$ form a \emph{cyclic triple}. 
In this case, for any points $p_i \in l_i$, $i=1,2,3$, there exists a Jordan curve $p_1p_2p_3p_1$ intersecting the leaves $l_1$, $l_2$, $l_3$ only at the points $p_i$, $i=1,2,3$, respectively. 
If this curve, from $p_1$ to $p_2$ to $p_3$ and returning to $p_1$ spins counterclockwise (resp. clockwise) we say that this cycle is \emph{positive} (resp. \emph{negative}) and denote this by $|l_1,l_2,l_3|^+$ (resp. $|l_1,l_2,l_3|^-$). 
The notation $|l_1,l_2,l_3|^{\pm}$ shall mean $|l_1,l_2,l_3|^+$ or $|l_1,l_2,l_3|^-$. 
By a \emph{chordal relation} of a triple of three different leaves $l_1, l_2, l_3 \in \mathscr{F}$ we will mean $l_1|l_2|l_3$, $l_2|l_1|l_3$, $l_1|l_3|l_2$, $|l_1,l_2,l_3|^{+}$ or $|l_1,l_2,l_3|^{-}$. 
For a very detailed and more general presentation of this, see \cite[Section 2]{K}. 

Given two subsets $U$ and $V$ of leaves of two regular foliations of the plane, we say that $U$ and $V$ are \emph{isomorphic} if there exists a bijection $g:U \to V$ such that for any triple $l_1, l_2, l_3 \in U$, the chordal relations of $l_1, l_2, l_3$ and $g(l_1), g(l_2), g(l_3)$ are the same. 
We say that $U$ and $V$ are \emph{anti-isomorphic} if there is a bijection $g: U\to V$ 
as above preserving all the chordal relations of triples of leaves but the signal of the cycles, that are reversed by $g$. 
The main result of Kaplan in \cite{K1} says that two regular foliations are o-topologically equivalent if and only if the set of leaves of them are isomorphic. 
Indeed, the result is quite more precise, see \cite[Theorem 26]{K1}: 
\begin{theorem}[\cite{K1}]\label{KK} 
If $\mathscr{F}_1$ and $\mathscr{F}_2$ are o-topologically equivalent regular foliations of the plane with o-equivalence homeomorphism $h$, then $h$ defines an isomorphism between $\mathscr{F}_1$ and $\mathscr{F}_2$. 
Conversely, if $g$ is an isomorphism between $\mathscr{F}_1$ and $\mathscr{F}_2$, then there exists an o-equivalence homeomorphism $h$ providing the o-topological equivalence of $\mathscr{F}_1$ and $\mathscr{F}_2$, and such that $h(l) = g(l)$ for all $l\in \mathscr{F}_1$. 
\end{theorem}

L. Markus \cite{markus-54} proved that in order to attest the o-topological equivalence between two foliations, it is only necessary to provide an isomorphism between the sets of separatrices plus one leaf in each canonical region. 
In order to state Markus's result we first define 
 
\begin{definition}\label{sepcon} 
A \emph{separatrix configuration} of a given regular foliation $\mathscr{F}$ is the set of separatrices $\mathfrak{S}_\mathscr{F}$ together with one leaf chosen in each canonical region. 
We denote the separatrix configuration of $\mathscr{F}$ by $S\mathfrak{S}_{\mathscr{F}}$. 
\end{definition}
Two separatrix configurations of a given regular foliation are isomorphic, see again \cite{markus-54}, so we can say \emph{the} separatrix configuration of a foliation. 
Here is Markus mentioned result: 
\begin{theorem}[\cite{markus-54}]\label{teo-markus} 
Two regular foliations $\mathscr{F}_1$ and $\mathscr{F}_2$ are o-topologically equivalent if and only if the separatrix configurations $S\mathfrak{S}_{\mathscr{F}_1}$ and $S\mathfrak{S}_{\mathscr{F}_2}$ are isomorphic. 
They are topologically equivalent if and only if the separatrix configurations are isomorphic or anti-isomorphic. 
\end{theorem}
This theorem says that the separatrix configuration is a complete topological invariant of a regular foliation. 
By the \emph{trivial foliation} we mean any representative of the class of foliations without separatrices. 

Before presenting the last result of this subsection, we introduce a notation: given $l_1$ and $l_2$ two leaves of $\mathscr{F}$, we define by $[l_1, l_2]$ (resp. $(l_1, l_2)$) to be the closed (resp. open) connected set whose boundary is $l_1\cup l_2$. 
Further $(l_1, l_2] = (l_1, l_2) \cup l_2 = [l_1, l_2]\setminus l_1$. 

The following result is Theorem 1 of Markus's paper \cite{markus-72} for a regular foliation $\mathscr{F}$ given by the orbits of a noncritical vector field of the plane. 
After a glance in its proof, one concludes that it is valid for any regular foliation of $\R^2$, by possibly using \cite[Theorem 30]{K} to conclude the necessity of statement \ref{222}. 

\begin{theorem}\label{Markuschara}
Let $\mathscr{F}$ be a regular foliation of the plane. 
\begin{enumerate}
\item A leaf $l \in \mathscr{F}$ is a separatrix if and only if for any pair of leaves $l_1, l_2 \in \mathscr{F}$ such that $l_1| l |l_2$, it follows that $[l_1, l_2]$ contains a cyclic triple of solutions. 
\item\label{222} On the other hand a leaf $l \in \mathscr{F}$ is not a separatrix if and only if there exist leaves $l_1, l_2\in \mathscr{F}$ with $l_1| l |l_2$ such that $[l_1, l_2]$ is the topological image of the plane strip $[0,1]\times \R$ with the curves of $\mathcal{F}|_{[l_1, l_2]}$ corresponding to the vertical lines of the strip.
\end{enumerate}
\end{theorem}

\subsection{Bifurcation sets} 
The topology of the fiber $f^{-1}(y)$ of a given continuous function $f$ may change when $y$ varies. 
As our examples in the introduction section, these changes can occur even when $y$ do not cross a singular value. 
The set where this changes happen is the \emph{bifurcation set}. 
More precisely: 

\begin{definition}\label{definition:tf} 
We say that a continuous map $p:X\to Y$ between topological spaces $X$ and $Y$ is a \textit{trivial fibration} if there exists a topological space $F$ and a homeomorphism $\varphi: F\times Y\to X$ such that $p\circ \varphi$ is the trivial projection. 
We say further that $p$ is a \textit{locally trivial fibration at} $y\in Y$ if there is a neighbourhood $U$ of $y$ such that the restriction $p^{-1}(U) \xrightarrow{\phantom{.} p \phantom{.}} U$ is a topologically trivial fibration. 
When $y\in Y$ does not satisfy this property, we say that $y$ is a \textit{bifurcation value}. 
We denote by $B(p)$ the set of bifurcation values of $p$. 
The set $B(p)$ is also called the \textit{bifurcation set} of $p$. 
In case $B(p) = \emptyset$ we say that $p$ is a \emph{locally trivial fibration}. 
\end{definition}

Notice that if $p$ is a locally trivial fibration, then its fibers are pairwise homeomorphic. 
It is known that if $p:\mathbb K^n\to\mathbb K$ is a polynomial function, where $\mathbb{K} = \R$ or $\C$, then $B(p)$ is a finite set (see \cite[Cor. 1.2.14]{T}). 

It is easy to check that if $p, q: X \to Y$ are topologically equivalent, i.e., if $q = \ell\circ p \circ h^{-1}$, with $\ell: Y \to Y$ and $h: X \to X$ homeomorphisms, then $p$ is a locally trivial fibration at $y\in Y$ if and only if $q$ is a locally trivial fibration at $\ell(y)$. 
So we have 
\begin{lemma}\label{prop-ida-1} 
Let $p,\,q: X\to Y$ satisfying $\ell \circ p = q \circ h$, with $\ell: Y \to Y$ and $h: X \to X$ homeomorphisms. 
Then $B(q) = \ell(B(p))$. 
In particular, if $p, q: \R^2 \to \R$ are o-topologically equivalent functions then $B(p)$ and $B(q)$ are homeomorphic with an increasing homeomorphism. 
\end{lemma}

\section{Linear-like functions}\label{section:linear}
In this section we study the linear like functions $p(x,y) = r(x) + s(x) y$ as defined in \eqref{linearform}. 
We pay attention to the submersion ones, studying the foliation $\mathscr{F}(p)$. 
We first describe the connected components of the level sets $L_c$ of $p$, i.e., the leaves of $\mathscr{F}(p)$. 
This provides a subdivision of $\R^2$ in certain vertical strips where the behavior of $\mathscr{F}(p)$ is described. 
We study the inseparable leaves of $\mathscr{F}(p)$ in Proposition \ref{prop-principal} and the limit separatrices in Corollary \ref{92345}. 
Then we systematically study the distinct behaviors of the foliation in each of these mentioned strips. 
In particular we will see that the behavior of $\mathscr{F}(p)$ in each of these strips depends only on the separatrices and their ``asymptotic'' behaviour, as defined later. 

We will end the section providing the bifurcation set of a finite linear-like submersion in Theorem \ref{bifurcation} and concluding that finite linear-like submersion functions do not present the vanishing at infinity phenomenon. 

It is straightforward to characterize when a linear-like function is a submersion, as well as to describe its levels sets: 
\begin{lemma}\label{submersion} 
Let $p$ be a linear-like function as in \eqref{linearform}. 
Then 
\begin{enumerate}[label={\textnormal{(\alph*)}}]
\item $p$ is a submersion if and only if $s(x)\neq 0$ $\forall x\in \R$ or each zero $z$ of $s(x)$ is not simple and $r'(z) \neq0$. 
\item For each $c\in \R$, the level set $L_c = p^{-1}(c)$ of $p$ is given by
$$
L_c = \left\{\big(x, N_c(x)\big) \in\mathbb{R}^2 \ | \ s(x)\neq0\right\}\cup\left\{(x,y)\in\mathbb{R}^2\ |\  s(x)=0,\, r(x)=c \right\}, 
$$ 
where 
$$
N_c(x) = \frac{c-r(x)}{s(x)}.
$$ 
\end{enumerate}
\end{lemma}

From now on, $\nu_a$ will stand for the vertical straight line $x=a$, $a\in \R$, i.e., 
$$
\nu_a = \{a\}\times \R. 
$$
Also, we will denote 
$$
Z_p = \{x \in \R\ |\ s(x) = 0\}. 
$$ 
Since this is a closed subset of $\R$, it follows that 
$$
\R\setminus Z_p = \cup_{i=1}^\infty (a_i, b_i)
$$ 
is a disjoint countable many union of (some possible empty) open intervals $(a_i, b_i)$, where $a_i$ could be $-\infty$ and $b_i$ could be $\infty$. 
So the leaves of $\mathscr{F}(p)$ are completely described as being of two types: 
\begin{quote} 
In each non-empty strip $(a_i, b_i) \times \R$ there exists exactly one connected component of $L_c$ given by the graphic of the function $N_c$ restricted to $(a_i, b_i)$, for any $c \in \R$. 
For each $z\in Z_p$ the vertical straight line $\nu_z$ is a connected component of $L_c$, where $c = r(z)$. 
\end{quote}

\begin{proposition}\label{prop-principal}
Let $p$ be a linear-like submersion such that $Z_p$ is not empty. 
If $a_i>-\infty$ (resp. $b_i < \infty$) then there exists exactly one leaf of $\mathscr{F}(p)$ contained in the strip $(a_i, b_i) \times \R$ which is inseparable to $\nu_{a_i}$ (resp. $\nu_{b_i}$). 
More precisely, letting $c$ be the value assumed by $p$ in $\nu_{a_i}$ (resp. $\nu_{b_i}$), then the graphic of the function $N_c$ defined in $(a_i, b_i)$ is inseparable to $\nu_{a_i}$ (resp. $\nu_{b_i}$). 

The collection of the above leaves are the inseparable leaves of $\mathscr{F}(p)$. 
\end{proposition} 
\begin{proof} 
To conclude that $\nu_{a_i}$ (resp. $\nu_{b_i}$) and the leaf in $(a_i, b_i)\times \R$ in the same level set of $\nu_{a_i}$ (resp. $\nu_{b_i}$) are inseparable, we provide two proofs. 
The first one is a direct application of \cite[Corollary 3.3]{BST}: 
``Let $f:\mathbb{R}^2\to\mathbb{R}$ be a $C^{\infty}$ submersion. Let $\gamma_1$ and $\gamma_2$ be two distinct connected components of a level set $f^{-1}(c)$. If $\lambda$ is an injective $C^{\infty}$ curve contained in $[\gamma_1, \gamma_2]$ connecting $\gamma_1$ to $\gamma_2$, then there exists at least a pair of inseparable leaves of $\mathscr{F}(f)$ contained in $[\gamma_1, \gamma_2]$ intercepting $\lambda$'' and the fact that two leaves inseparable to each other are connected components of the same level set of $p$. 

The second one is as follows. 
Let $c \in \R$ be the value taken by $p$ in $\nu_{a_i}$ and denote by $\gamma_c$ the only connected component of the level set $L_c$ contained in $(a_i, b_i) \times \R$, i.e., the graphic of the function $N_c$ in $(a_i, b_i)$. 

Let $S_1$ be a transversal section through $\gamma_c$ and $S_2$ be the intersection of a transversal section through the straight line $\nu_{a_i}$ with the strip $[a_i, b_i) \times \R$. 

The images $p(S_1)$ and $p(S_2)$ are intervals containing $c$ such that for $\varepsilon>0$ small enough the first one contains the interval $(c-\varepsilon, c+\varepsilon)$ and the second one the interval $(c-\varepsilon, c)$ or $(c, c+ \varepsilon)$. 
In any case, there exists $d\neq c$ with $d \in p(S_1) \cap p(S_2)$. 

The level set $L_d$ intersects $S_1$ and $S_2$  by construction.  
Since there exists exactly one connected component of $L_d$ in the strip $(a_i, b_i)\times \R$, this connected component must intersect $S_1$ and $S_2$, proving that $\gamma_c$ and $\nu_{a_i}$ are inseparable. 

Now we prove that there are no other type of inseparable leaves. 
Since in any strip $(a_i, b_i) \times \R$, there is only one connected component of each level set of $p$, any leaf in such a strip can only be inseparable to $\nu_{a_i}$ or to $\nu_{b_i}$. 

So let $\nu_{z_1}$ and $\nu_{z_2}$, $z_1, z_2 \in Z_p$, be two leaves inseparable to each other. 
There does not exist $z\in Z_p$ between $z_1$ and $z_2$ otherwise the leaf $\nu_z$ would separate $\nu_{z_1}$ from $\nu_{z_2}$. 
So $z_1, z_2 \in \{a_i,b_i\}$ for some $i$, finishing the proof. 
\end{proof}

The following is a direct consequence of Proposition \ref{prop-principal}. 
\begin{corollary}\label{92345}
Let $p$ be a linear-like submersion such that $Z_{p}\neq \emptyset$. 
The leaves $\nu_z$ such that $z \in \overline{\cup_{i=1}^{\infty}\{a_i, b_i\}}\setminus\cup_{i=1}^{\infty}\{a_i, b_i\}$ are the limit separatrices of $\mathscr{F}(p)$. 
\end{corollary}

Next we describe the behavior of $\mathscr{F}(p)$ in each nonempty strip $(a_i, b_i)\times \R$. 
We have the following possible configurations of separatrices and canonical regions according to Proposition \ref{prop-principal}: 
\begin{enumerate}[label={\textnormal{(\Roman*)}}]
\item\label{wa} If $a_{i} = -\infty$ (resp. $b_i = \infty$), we have one inseparable leaf and so exactly two canonical regions of $\mathscr{F}(p)$ contained in $(a_{i}, b_i) \times \R$.  
\item\label{wb} If $- \infty < a_{i} < b_i < \infty$ and the values that $p$ assumes in $\nu_{a_{i}}$ and in $\nu_{b_i}$ are distinct, we have two inseparable leaves and so three canonical regions of $\mathscr{F}(p)$ contained in $(a_{i}, b_i) \times \R$
\item\label{wc} If $- \infty < a_{i} < b_i < \infty$ and the values that $p$ assumes in $\nu_{a_{i}}$ and in $\nu_{b_i}$ are equal, we have one inseparable leaf and so two canonical regions of $\mathscr{F}(p)$ contained in $(a_{i}, b_i) \times \R$. 
\end{enumerate} 

In order to determine the foliation in each strip, we will act according to Theorem \ref{teo-markus}, describing the separatrix configuration in each strip. 
This will be done by determining the behavior of a leaf in each of the canonical regions in configurations \ref{wa}, \ref{wb} and \ref{wc}. 
We first study the behavior of a leaf in a canonical region having in its boundary at least two leaves inseparable to each other. 
By Proposition \ref{prop-principal} and the possible configurations \ref{wa}, \ref{wb} and \ref{wc} above, this boundary must contain a straight line $\nu_{a_i}$ or $\nu_{b_i}$, and a connected component of the level set $L_c$, where $c$ is the value that $p$ takes in $\nu_{a_i}$ or in $\nu_{b_i}$, respectively. 
We denote by $\gamma_c$ this connected component, which is given by the graphic of the function $N_c$ restricted to $(a_{i}, b_i)$. 
Observe that the straight lines $\nu_{a_i}, \nu_{b_i}$ are vertical asymptotes of $N_c$. 
\begin{lemma}\label{prop-comportamento}
Any leaf $\gamma$ in the canonical region whose boundary contains $\nu_{a_i}$ and $\gamma_c$ satisfies $|\nu_{a_i}, \gamma, \gamma_c|^{\pm}$ if $\lim\limits_{x\to {a_i}^{+}} N_c(x)=\pm \infty$. 
Any leaf $\gamma$ in the canonical region whose boundary contains $\nu_{b_i}$ and $\gamma_c$ satisfies $|\nu_{b_i}, \gamma, \gamma_c|^\mp$ if $\lim\limits_{x\to a^{-}_{i}}N_c(x)=\pm\infty$. 
\end{lemma} 
\begin{proof} 
We prove the case where $\gamma$ is in the canonical region whose boundary contains $\nu_{a_i}$ and $\gamma_c$, and $\lim\limits_{x\to {a_i}^{+}} N_c(x)= + \infty$. 
All the other situations have a completely similar proof. 

The leaf $\gamma$ is in particular contained in $(\gamma_c, \nu_{a_i})$. 
Moreover, $\gamma$ cannot separates $\gamma_c$ and $\nu_{a_i}$, otherwise these two leaves will not be inseparable. 
So $|\nu_{a_i}, \gamma, \gamma_c|^+$ or $|\nu_{a_i}, \gamma, \gamma_c|^-$. 

Let $p_1 \in \nu_{a_i}$, $p_2\in \gamma$ and $p_3 \in \gamma_c$ and consider a Jordan closed curve $p_1p_2p_3p_1$ meeting these three leaves only in $\{p_1, p_2,p_3\}$ and oriented from $p_1$ to $p_2$ to $p_3$ to $p_1$. 

We know that the leaf $\gamma$ is also the graphic of a function $N_d$ defined in $(a_{i}, a_{i+1})$ with $\lim\limits_{x\to a_i^+} N_d(x) = \pm \infty$ and such that $N_d(x) \leq N_c(x)$ for all $x \in (a_i, a_{i+1})$. 
We must have $\lim\limits_{x\to a_i^+} N_d(x) = - \infty$, otherwise $\nu_{a_i}$ and $\gamma_c$ will not be inseparable. 
So the curve $p_1 p_2 p_3 p_1$ is counterclockwise oriented, otherwise $\gamma$ will cut this curve at least twice. 
So $|\nu_{a_i}, \gamma, \gamma_c|^+$. 
\end{proof} 

\subsection{Configuration \ref{wa}}
If $a_{i} = -\infty$, let $c_{b_i}$ be the value taken by $p$ in $\nu_{b_i}$ and $\gamma_{c_{b_i}}$ be the leaf given by the only connected component of the level set $L_{c_{b_i}}$ in $(a_{i}, b_i)$. 
This curve separates the plane into two open regions. 
One of them does not contain any separatrix and so will determine one of the canonical regions. 
A given leaf $\gamma$ in this canonical region has just one possible behavior. 
Precisely, by considering $N_d(x)$ the function whose graphic is $\gamma$, $\lim\limits_{x\to b_i^-}N_d(x) = \lim\limits_{x \to b_i^-}N_{c_{b_i}}(x)$. 
The other canonical region is contained in the opposite open region. 
It has as boundary $\gamma_{c_{b_i}}$ and $\nu_{b_i}$, so by Lemma \ref{prop-comportamento} it follows that $|\nu_{b_i}, \gamma, \gamma_{c_{b_i}}|^\mp$ if $\lim\limits_{x\to {b_i}^{-}}N_{c_{b_i}}(x)=\pm\infty$. 
See the separatrix configuration of this strip in case $\lim\limits_{x\to b_i^-} N_{c_{b_i}}(x) = +\infty$ in (a) of Fig. \ref{possibilities3}. 
\begin{figure}[h!]
\begin{center}
\subfigure[$a_{i} = -\infty$]{\includegraphics[scale=0.7]{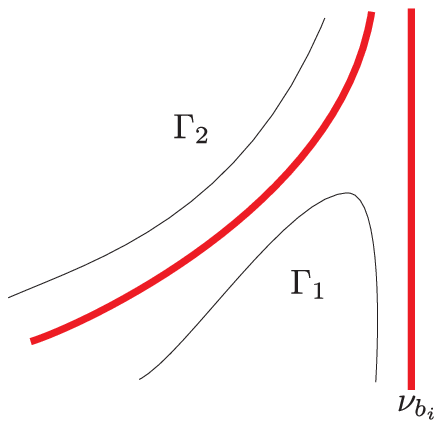}}\ \ \ \ \ \ \ \ \ \ \ \ \ \ \ \ \ \ \ \ 
\subfigure[$b_i = \infty$]{\includegraphics[scale=0.7]{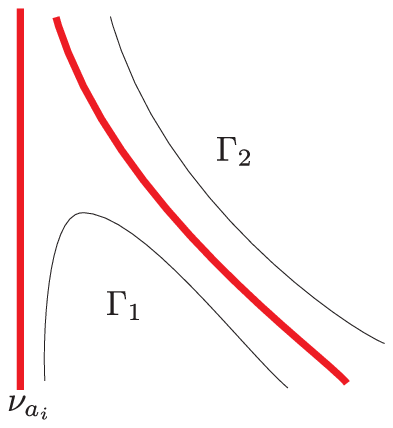}}
\end{center}
\caption{\small Separatrix configurations on the strip $(a_{i}, b_i] \times \R$ when $a_{i}=-\infty$ and on the strip $[a_{i}, b_i) \times \R$ when $b_i = +\infty$.}\label{possibilities3}
\end{figure}

If $b_i =\infty$, by letting $c_{a_i}$ the value $p$ takes in $\nu_{a_{i}}$, we have a completely analogous analysis: in the canonical region having as boundary the curve $\gamma_{c_{a_i}}$ there is only one possibility. 
While in the canonical region whose boundary is $\nu_{a_{i}}\cup \gamma_{c_{a_i}}$ it follows from Lemma \ref{prop-comportamento} that $|\nu_{a_{i}}, \gamma, \gamma_{c_{a_i}}|^\pm$ if $\lim\limits_{x\to  {a_{i}}^{+}}N_{c_{a_i}}(x)=\pm\infty$. 
See in (b) of Fig. \ref{possibilities3} a depiction in case $\lim\limits_{x\to {a_{i}}^{+}}N_{c_{a_i}}(x)=+\infty$. 

\subsection{Configuration \ref{wb}} 
Let $c_{a_i}$ and $c_{b_i}$ be the values taken by $p$ in $\nu_{a_{i}}$ and $\nu_{b_i}$, and $\gamma_{c_{a_i}}$ and $\gamma_{c_{b_i}}$ be the graphics of $N_{c_{a_i}}$ and $N_{c_{b_i}}$ in $(a_i, b_i)$, respectively. 

We assume $\lim\limits_{x\to a^{+}_{i}}N_{c_{a_i}}(x)=+\infty$, the other possibility is completely analogous. 
We have eight possibilities a priori for $\lim\limits_{x\to a_{i}^+} N_{c_{b_i}}(x)$, $\lim\limits_{x\to b_i^-} N_{c_{a_i}}(x)$ and $\lim\limits_{x\to b_{i}^-} N_{c_{b_i}}(x)$. 
But it is simple to conclude that, in order not to break the inseparability of $\nu_{a_{i}}$ with $\gamma_{c_{a_i}}$ and of $\nu_{b_i}$ with $\gamma_{c_{b_i}}$, as well as considering that different leaves cannot intersect each other, the only possibilities are the following four: \\ 
If $\lim\limits_{x\to a^{+}_{i}}N_{c_{b_i}}(x)=+\infty$, 
\begin{enumerate}[label={{{(\alph*)}}}]
\item $\lim\limits_{x\to b^{-}_{i}}N_{c_{a_i}}(x)=-\infty$ and $\lim\limits_{x\to b^{-}_{i}}N_{c_{b_i}}(x)=+\infty$, 
	
\item $\lim\limits_{x\to b^{-}_{i}}N_{c_{a_i}}(x)=-\infty$ and $\lim\limits_{x\to b^{-}_{i}}N_{c_{b_i}}(x)=-\infty$. 
\end{enumerate}
\noindent 
And, if $\lim\limits_{x\to a^{+}_{i}}N_{c_{b_i}}(x)=-\infty$, 

\begin{enumerate}[label={{{(\alph*)}}},resume]
	
\item $\lim\limits_{x\to b^{-}_{i}}N_{c_{a_i}}(x)=+\infty$ and $\lim\limits_{x\to b^{-}_{i}}N_{c_{b_i}}(x)=+\infty$, 
	
\item $\lim\limits_{x\to b^{-}_{i}}N_{c_{a_i}}(x)=+\infty$ and $\lim\limits_{x\to b^{-}_{i}}N_{c_{b_i}}(x)=-\infty$. 
\end{enumerate} 
See Figure \ref{possibilities}. 
\begin{figure}[h!]
\begin{center}
\subfigure[]{\includegraphics[scale=0.7]{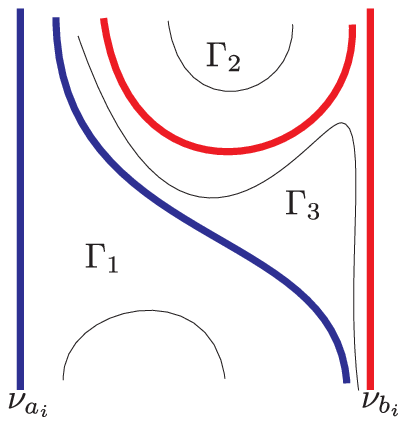}}\ \ \ \ \ \ \ \ \ \ \ \ \ \ \ \ \ \ \ \
\subfigure[]{\includegraphics[scale=0.7]{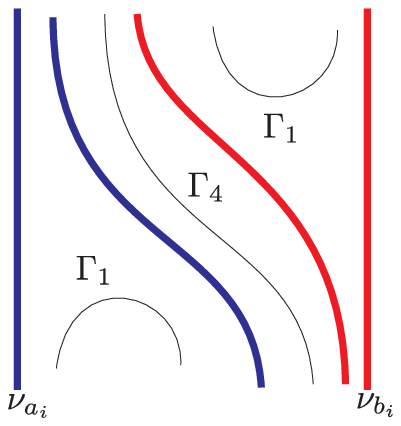}}\\
\subfigure[]{\includegraphics[scale=0.7]{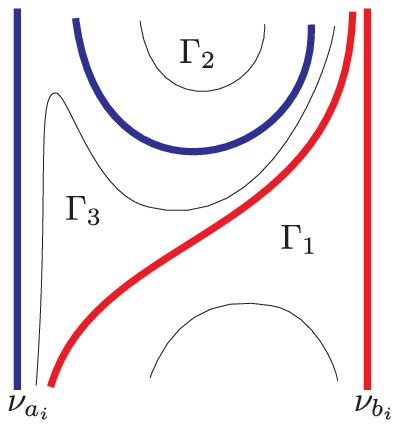}}\ \ \ \ \ \ \ \ \ \ \ \ \ \ \ \ \ \ \ \
\subfigure[]{\includegraphics[scale=0.7]{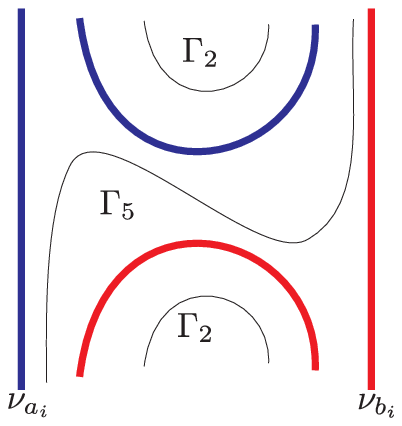}}	
\end{center}
\caption{\small Separatrix configurations in $[a_{i}, b_i] \times \R$ when $c_{a_i}\neq c_{b_i}$ in case $\lim\limits N_{c_{a_i}}(x) = + \infty$. 
The blue curve is $\gamma_{c_{a_i}}$ and the red one is $\gamma_{c_{b_i}}$. 
Completely analogous situations will happen when $\lim\limits N_{c_{a_i}}(x) = - \infty$.}\label{possibilities}
\end{figure} 
Now we have to add one leaf in each of the three canonical regions for each of these four cases. 

For the canonical regions having in their boundaries at least two inseparable leaves, the only possibility is given by Lemma \ref{prop-comportamento} and the fact that the leaves are graphics of functions: see regions labeled as $\Gamma_1$, $\Gamma_3$ and $\Gamma_5$ in Fig. \ref{possibilities}. 
Each of the other canonical regions in (a), (b), (c) and (d) can be of two types: or it has only one leaf in its boundary, labeled as $\Gamma_2$ in Fig. \ref{possibilities}, or it has exactly two leaves that are not inseparable to each other in its boundary, labeled as $\Gamma_4$ in Fig. \ref{possibilities}. 
In both cases, a leaf in such a region, being the graphic of a function defined in $(a_{i}, a_i)$, must follows the behavior of $N_{c_{a_i}}$ and $N_{c_{b_i}}$. 
So the only possibility is as given in each case of Fig. \ref{possibilities}. 

\subsection{Configuration \ref{wc}}
Let $c_{a_i}$ be the common value $p$ attains in $\nu_{a_{i}}$ and in $\nu_{b_i}$, and let $\gamma_{c_{a_i}}$ be the graphic of $N_{c_{a_i}}$ in $(a_{i}, b_i)$. 
We assume $\lim\limits_{x\to a^{+}_{i}}N_{c_{a_i}}(x)=+\infty$. 
We distinguish two cases: 
\begin{enumerate}[label={{{(\alph*)}}}]
\item\label{p11} $\lim\limits_{x\to b^{-}_{i}}N_{c_{a_i}}(x)=+\infty$, which is equivalent to $|\nu_{a_{i}}, \gamma_{c_{a_i}}, \nu_{b_{i}}|^-$, 
\item\label{p12} $\lim\limits_{x\to b^{-}_{i}}N_{c_{a_i}}(x)=-\infty$, which is equivalent to $\nu_{a_{i}}| \gamma_{c_{a_i}}| \nu_{b_{i}}$. 
\end{enumerate}
(With the assumption $\lim\limits_{x\to a^{+}_{i}}N_{c_{a_i}}(x)=-\infty$, the $+$ and $-$ in \ref{p11} and in \ref{p12} must be exchanged and the reasoning is similar.)

In case \ref{p11}, Lemma \ref{prop-comportamento} and the fact that leaves are graphics of functions provide the behavior of any leaf in the canonical region bounded by $\nu_{a_{i}}$, $\gamma_{c_{a_i}}$ and $\nu_{b_{i}}$. 
The other canonical region has only $\gamma_{c_{a_i}}$ as its boundary, so any leaf there has its behavior as $\gamma_{c_{a_i}}$. 
See (a) of Fig. \ref{possibilities2}. 
\begin{figure}[h!]
\begin{center}
\subfigure[]{\includegraphics[scale=0.7]{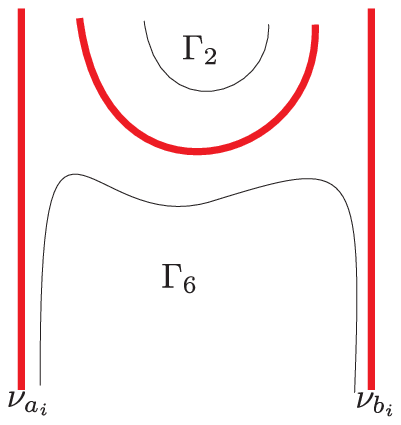}}\ \ \ \ \ \ \ \ \ \ \ \ \ \ 
\subfigure[]{\includegraphics[scale=0.7]{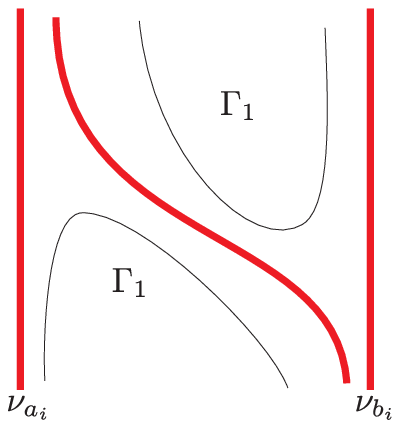}}
\end{center}
\caption{\small Separatrix configurations in the strip $[a_{i}, b_i]\times \R$ when $p$ attains the same value $c_{a_i}$ in $\nu_{a_{i}}$ and in $\nu_{b_i}$, in case $\lim\limits_{x\to a_{i}^+}N_{c_{a_i}}(x) = +\infty$. 
Completely similar situations occur when $\lim\limits_{x\to a_{i}^+}N_{c_{a_i}}(x) = -\infty$.}\label{possibilities2}
\end{figure}
In case \ref{p12}, both canonical regions will be determined by Lemma \ref{prop-comportamento}. 
See (b) in Fig. \ref{possibilities2}. 

So the separatrix configuration of $\mathscr{F}(p)$ in each strip $[a_{i}, b_i] \times \R$ is completely characterized by the ``asymptotic'' behavior of the separatrices in $(a_{i}, b_i) \times \R$, i.e., in order to determine it we only need to know whether these separatrices, which are graphics of functions $N_c(x)$ defined in $(a_i, b_i)$, where $c$ is the value $p$ takes in $\nu_{a_{i}}$ and/or $\nu_{b_i}$, are such that the limit when $x$ tends to $a_{i}$ and/or $b_i$ is $+\infty$ or $-\infty$. 
So in some sense the separatrix configuration in a strip region is determined by the configuration of the separatrices in it: the behavior of the ordinary leaves in each canonical region are determined by the behavior of the separatrices according to the configurations \ref{wa}, \ref{wb} and \ref{wc}. 
A similar behavior may not happen for functions other than linear-like ones; indeed, in Example \ref{676767} we present a submersion function where there are ordinary leaves ``misbehaving'', i.e. they do not respect Lemma \ref{prop-comportamento}. 

\smallskip

Therefore, in order to determine the separatrix configuration of $\mathscr{F}(p)$ for a given linear-like submersion $p$, we first find the pairwise disjoint open intervals $(a_i, b_i)$ such that $\R \setminus Z_p = \cup_{i} (a_i, b_i)$. 
Then we analyze the asymptotic behavior of the functions $N_c(x)$ in the intervals $(a_i, b_i)$ where $c$ is the value that $r$ takes in $a_i$ and/or $b_i$. 
Since $\nu_{a_i}$ and $\nu_{b_i}$ are vertical asymptotes of $N_c(x)$, it follows that the limit of $N_c(x)$ when $x\to a_i^+$ (resp. $x\to b_i^-$) is $\pm \infty$ if and only if the signal of $N_c(x)$ with $x>a_i$ close to $a_i$ (resp. $x<b_i$ close to $b_i$) is $\pm$. 
According to configurations \ref{wa}, \ref{wb} and \ref{wc} this will determine the separatrix configuration in each strip $[a_i, b_i]\times\R$. 

Further, we consider the straight lines $\nu_z$, where $z\in \overline{\cup_{i}\{a_i, b_i\}}\setminus \cup_i \{a_i, b_i\}$, which will be the limit separatrices. 
Finally in each connected component of the complement of $\overline{\cup_{i}(a_i, b_i)}$, we take an $x$. 
It follows that each canonical region away from $\cup_i (a_i, b_i)\times \R$ contains a straight line $\nu_x$. 
Hence our separatrix configuration is complete. 

We end this section by determining the bifurcation set of a finite linear-like submersion $p$. 
In this case, by writing $Z_p = \{a_1, a_2,\ldots, a_k\}$, $a_i<a_{i+1}$, $i=1,\ldots, k-1$, we have that the intervals that cover $\R\setminus \Z_p$ as above are of the form $(a_i, a_{i+1})$, $i= 0, \ldots, k$, where $a_0 = -\infty$ and $a_{k+1} = \infty$. 

\begin{theorem}\label{bifurcation}
Let $p(x,y) = r(x) +s(x) y$ be a finite linear-like submersion. 
Then $B(p) = r(Z_p)$. 
\end{theorem}
\begin{proof}
Let $Z_p = \{a_1, a_2,\ldots, a_k\}$ as above and $c \notin r(Z_p)$. 
By Lemma \ref{submersion} it follows that $p^{-1}(c)$ has $k+1$ connected components $\gamma_1, \ldots, \gamma_{k+1}$ such that $\gamma_i \subset (a_{i-1}, a_i) \times \R$, $i=1, \ldots k+1$, where $a_0= -\infty$ and $a_{k+1} = \infty$. 
Moreover, by Proposition \ref{prop-principal} it follows that none of the $\gamma_i$, $i=1, \ldots, k+1$, is a separatrix of $\mathscr{F}(p)$. 
So it follows from Theorem \ref{Markuschara} that for each $i = 1,\ldots, k+1$, there exist two leaves $\beta_i^1, \beta_i^2 \in \mathscr{F}(p)$ and $\varepsilon_i>0$ such that the closed region $[\beta_i^1, \beta_i^2]$ is mapped by a homeomorphism $h_i$ onto the strip $[c-\varepsilon_i, c + \varepsilon_i]\times \R$ such that the vertical lines correspond to the leaves of $[\beta_i^1, \beta_i^2]$, $i=1,\ldots, k+1$. 
We can assume further that $h_i^{-1}(\{d\}\times \R)$ is the connected component of $p^{-1}(d)$ in $[\beta_i^1, \beta_i^2]$. 
Let $\varepsilon = \min_{1\leq i\leq k+1}\{\varepsilon_i\}$ and define $h: F\times (c - \varepsilon, c + \varepsilon) \to p^{-1}(c-\varepsilon, c + \varepsilon)$, where $F = \{1,2,\ldots, k+1\} \times \R$, by $h((i,x),y) = h_i^{-1}(y,x)$. 
This is a homeomorphism and $p\circ h((i,x),y) = y$, proving by definition that $c\notin B(p)$, hence $B(p) \subset r(Z_p)$. 

Now let $c\in r(Z_p)$. 
By Lemma \ref{submersion}, $p^{-1}(c)$ has at least $k+2$ connected component whereas $p^{-1}(d)$ for any $d$ sufficiently close to $c$ has $k+1$ connected components. 
Therefore $c \in B(p)$ and hence $r(Z) = B(p)$. 
\end{proof}

Next corollary will use the terminology of M. Tib\u ar and A. Zaharia paper \cite{TZ}. 
\begin{corollary}\label{12091}
A finite linear-like submersion does not have the vanishing at infinity phenomenon. 
\end{corollary}
\begin{proof}
Let $c_0 \in \R$. 
If $c_0 \notin B(p)$, it follows by Theorem \ref{bifurcation} and Lemma \ref{submersion} that $p^{-1}(c)$ has the same number of connected components, each of them given by the graphic of the function $N_c(x)$ in each interval $(a_i, a_{i+1})$. 
If $c\to c_0^+$ or $c\to c_0^-$, it follows that each connected component of $p^{-1}(c)$ will tend to the corresponding connected component of $p^{-1}(c_0)$. 

Now if $c_0\in B(p)$, it follows by our analysis in each configuration \ref{wa}, \ref{wb} and \ref{wc} that each connected component of $p^{-1}(c)$ will split into two or three connected component of $p^{-1}(c_0)$ (one of them being $\nu_{a_i}$, for some $a_i\in Z_p$) or will tend to a connected component of $p^{-1}(c_0)$, when $c\to c_0^+$ or $c\to c_0^-$. 
So there are no connected components of $p^{-1}(c)$ vanishing at infinity when $c\to c_0^{\pm}$. 
\end{proof}

\section{Examples}\label{section:examples} 
We gather in this section the examples of the paper. 
We provide details of the just mentioned examples throughout the paper, as well as we present some other linear-like submersions $p$ and its foliation $\mathscr{F}(p)$. 
We include examples of linear-like submersions $p$ where the set $Z_p$ is not discrete. 
In this case, it can appear limit separatrices. 

We begin by detailing the examples of the introduction section. 

\begin{example}\label{degree2}
Let $p$ and $q$ be the functions defined in \eqref{ex122}. 
They are submersions because the derivatives of $r_1(x) = x (7 x - 5)$ and $r_2(x) = 2 x (4 x - 5)$ are not zero in the set $Z_p = Z_q = \{-1,0,2\}$, according to Lemma \ref{submersion}. 

We calculate in details the separatrix configuration of $\mathscr{F}(p)$ and let to the reader the calculations on $\mathscr{F}(q)$. 
According to Section \ref{section:linear} we have to study the asymptotic behavior of the functions $N_c(x)$, defined in Lemma \ref{submersion} for $p$, for $c= r_1(-1) = 12$, $r_1(0) = 0$ and $r_1(2) = 18$, on the domains $(-\infty, -1) \cup (-1,0)$, $(-1,0)\cup (0,2)$ and $(0,2) \cup (2, \infty)$, respectively. 
In order to do that, we only need to know the signals of $N_{12}(x)$ close to $x=-1$ and  close to $0$ (with $x<0$); of $N_{0}(x)$ close to $x=-1$ (with $x>-1$), close to $x=0$ and close to $x=2$ (with $x<2$); and of $N_{18}(x)$ close to $x=0$ (with $x>0$), and close to $x=2$. 
We conclude that the behaviors are like the ones in (a) of Fig. \ref{figura}. 
\begin{figure}[htpb]
\begin{center}
\subfigure[$S\mathfrak{S}_{\mathscr{F}(p)}$]{\includegraphics[scale=.58]{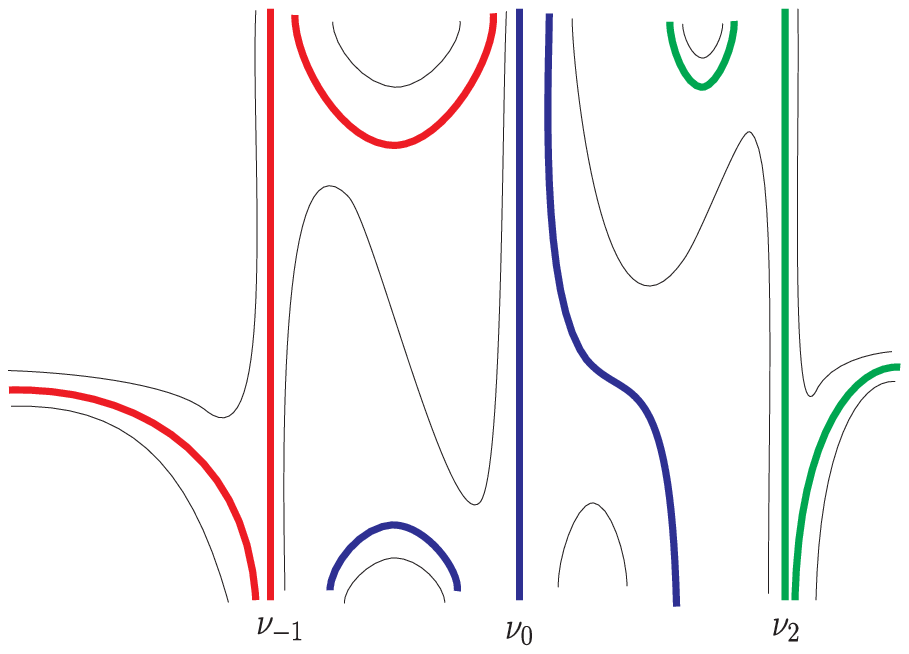}}
\ \ \  \ \ \ \ \
\subfigure[$S\mathfrak{S}_{\mathscr{F}(q)}$]{\includegraphics[scale=.58]{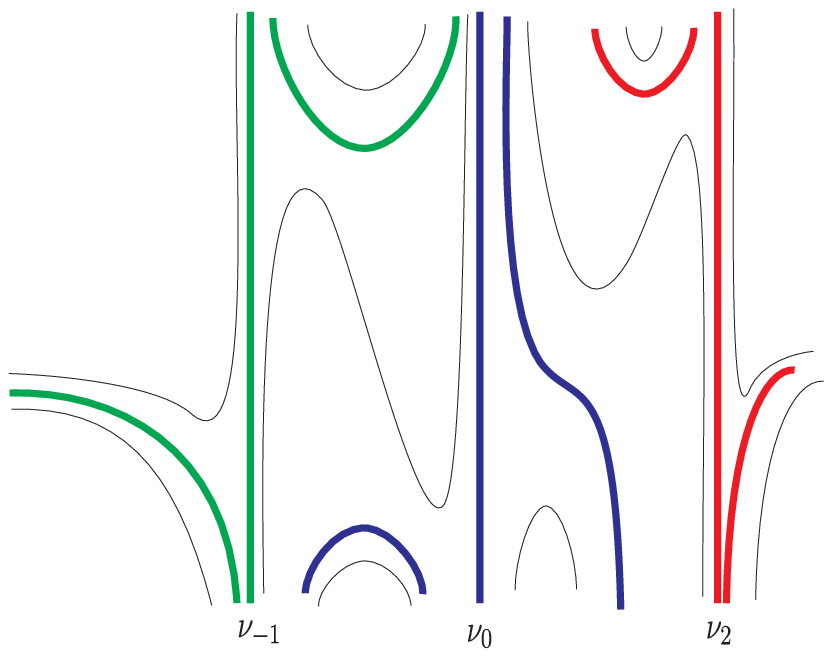}}
\end{center}
\caption{\small Separatrix configurations of $\mathscr{F}(p)$ and $\mathscr{F}(q)$. In red we depict the separatrices contained in the level set $L_{12}$, in blue the ones contained in $L_0$ and in green the ones in $L_{18}$.}\label{figura}
\end{figure}
According to Section \ref{section:linear} the ordinary leaves in each canonical region are automatically determined and so we get the separatrix configuration of $\mathscr{F}(p)$. 
Analogously, the reader can conclude that the separatrix configuration of $\mathscr{F}(q)$ is as presented in (b) of Fig. \ref{figura}. 

After a glance in $S\mathfrak{S}_{\mathscr{F}(p)}$ and $S\mathfrak{S}_{\mathscr{F}(q)}$ in Fig. \ref{figura}, we readily conclude that they are isomorphic and so, by Theorem \ref{teo-markus}, $\mathscr{F}(p)$ and $\mathscr{F}(q)$ are o-topologically equivalent. 

We now prove that $p$ and $q$ cannot be topologically equivalent. 

By Theorem \ref{bifurcation}, we have that $B(p) = B(q) = r_1(Z_p) = r_2(Z_q) = \{0, 12, 18\}$. 
If $p$ and $q$ were topologically equivalent, then by Corollary \ref{coro:main} there would exist an equivalence homeomorphism $h: \R^2 \to \R^2$ and a monotone bijection $\sigma: B(p)\to B(q)$ such that $p^{-1}(\sigma^{-1}(c)) = h^{-1} (q^{-1}(c))$, $c = 0, 12, 18$. 
Since $h$ carries leaves inseparable to each other onto leaves inseparable to each other and preserves chordal relations of triples, it follows that the set $\{\nu_{-1}$, $\nu_0, \nu_2\}$ must be carried onto itself, keeping  $\nu_0$ fixed. 
Since $\sigma$ is monotone, it then follows that $h(\nu_{-1}) = \nu_{2}$, $h(\nu_2) = \nu_{-1}$, and so $\sigma$ is the identity. 

Now, $h$ must take the two separatrices in the region bounded by $\nu_{-1}$ and $\nu_0$ in (a) of Fig. \ref{figura}  onto the two separatrices contained in the region bounded by $\nu_{0}$ and $\nu_2$ in (b) of Fig. \ref{figura}, keeping the colors. 
But then the cyclic triple formed by $\nu_{-1}$, $\nu_0$ and the blue separatrix between them in (a) of Fig. \ref{figura} would be taken to the non-cyclic triple $\nu_0$, $\nu_2$ and the blue separatrix between them in (b) of Fig. \ref{figura}, a contradiction because $h$ must preserve chordal relations. 
Therefore, $p$ and $q$ cannot be topologically equivalent. 

A final comment in this example is that for any $c\in \R$ the number of connected components of $p^{-1}(c)$ and $q^{-1}(c)$ is the same: for $c=0, 12, 18$, this number is $5$, and for any different $c$ this number is $4$. 
In particular the Euler characteristics of the generic fiber of $p$ and $q$ are the same. 
\end{example}

\begin{example}\label{dd45}
Let $p: \R^2 \to \R$ be the function given by 
$$
p(x,y) = x + x^3 y. 
$$ 
It is simple to see that this is a submersion. 
Acting as in Example \ref{degree2}, we obtain the separatrices of $\mathscr{F}(p)$ and $\mathscr{F}(-p)$, that are the same, by taking $\nu_0$ and the graphic of the function $N_0(x) = -x/x^3=-1/x^2$, with $x\neq 0$. 
The ordinary leaves are automatically given, according to Section \ref{section:linear}. 
It is clear that $\mathscr{F}(p)$ and $\mathscr{F}(-p)$ are o-topologically equivalent. 

On the other hand, we \emph{claim that $p$ and $-p$ cannot be topologically equivalent}. 
Indeed, the difference between the functions is exactly in their ordinary leaves, given by connected components of the level sets $L_c$ with $c\neq 0$: the level set $L_c$ of $p$ is the level set $L_{-c}$ of $-p$, see Fig. \ref{fig:fol-noeq}.  
\begin{figure}[h]
\begin{center}
\subfigure[$S\mathfrak{S}_{\mathscr{F}(p)}$]{\includegraphics[scale=0.60]{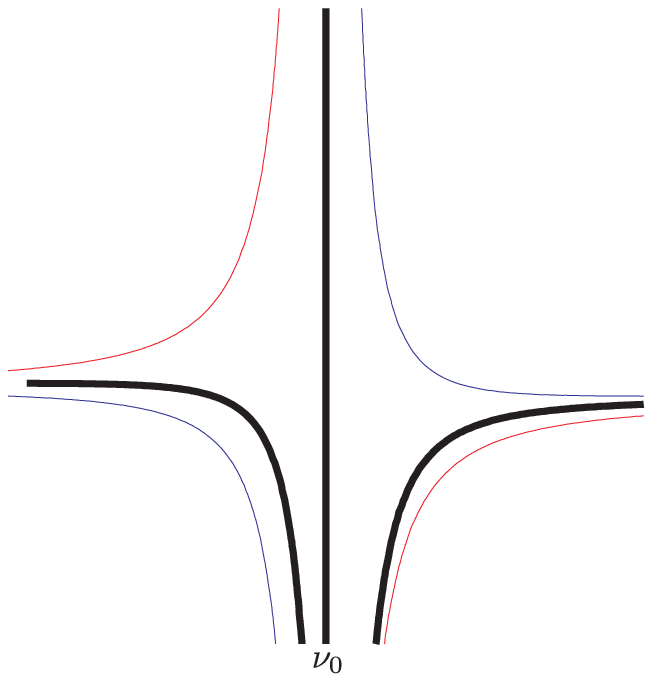}}\ \ \ \ \ \ \ \ \ \ \ \ \ 
\subfigure[$S\mathfrak{S}_{\mathscr{F}(-p)}$]{\includegraphics[scale=0.60]{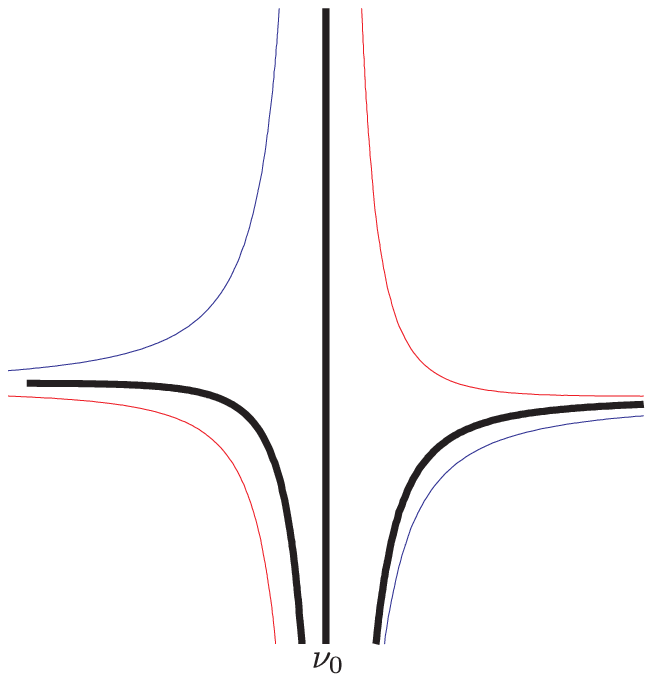}}
\end{center}
\caption{\small Separatrix configurations of $\mathscr{F}(p)$ and $\mathscr{F}(-p)$. 
The black lines are the connected components of the level set $L_0$, the inseparable leaves; the red ones are the connected component of a level set $L_t$, $t<0$, and the blue ones are connected components of a level set $L_t$ for $t>0$.}\label{fig:fol-noeq}
\end{figure} 
If there exist preserving orientation homeomorphisms $\ell: \R\to\R$ and $h: \R^2 \to \R^2$ such that $\ell\circ (-p) = p\circ h$ it follows in particular that $\ell(0) = 0$ (because $h$ sends separatrices onto separatrices) and so $h$ must take positive (resp. negative) level sets of $p$ onto positive (resp. negative) ones of $-p$. 
But since $h$ keeps chordal relations of triples, it must send so the cyclic triple formed by $\nu_0$, the connected components of $L_0$ to the left of $\nu_0$ and a red curve between them in (a) of Fig. \ref{fig:fol-noeq} onto the cyclic triple formed by $\nu_0$, the connected component of $L_0$ to the right of $\nu_0$ and a red curve between them in (b) of Fig. \ref{fig:fol-noeq}. 
But these cycles have reversed orientation. 
So $p$ and $-p$ cannot be o-topologically equivalent. 
\end{example}

The following example shows that the ``monotone'' assumption cannot be strengthened in corollaries \ref{coro:main}, \ref{varcorollary:main} and \ref{cor5.1}. 

\begin{example}\label{only} 
Let 
$$
p(x,y) = x (3 - 2 x) + (x - 1)^2 x^2 y, \ \ \ \ q(x,y) = (x - 1) (2 x - 1) - (x - 1)^2 x^2 y. 
$$
As above, these linear-like functions are submersions. 
Observe that 
$$
\ell\circ p = q, 
$$
with $\ell(t) = 1-t$, hence $p$ and $q$ are topologically equivalent. 

We have $Z_p = Z_q = \{0,1\}$, and $p(\nu_0) = q (\nu_1) = 0$ and $p(\nu_1) = q(\nu_0) = 1$. 
For the function $p$, by analyzing the signal of $N_0(x)$ near $0$ and in the interval $(0,1)$, as well as the signal of $N_1(x)$ in the interval $(0,1)$ and near $1$, we conclude, according to Section \ref{section:linear}, that the separatrix configuration of $\mathscr{F}(p)$ is as shown in (a) of Fig. \ref{ultima}. 
\begin{figure}[h]
\begin{center}
\subfigure[$S\mathfrak{S}_{\mathscr{F}(p)}$]{\includegraphics[scale=0.6]{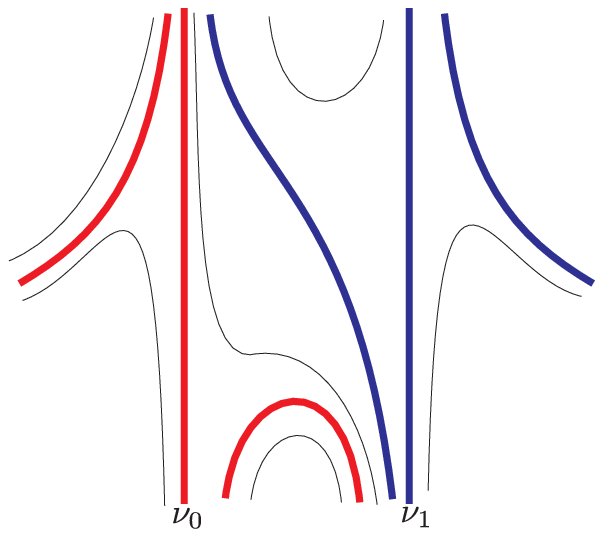}}\ \ \ \ \ \ \ \ \ \ \
\subfigure[$S\mathfrak{S}_{\mathscr{F}(q)}$]{\includegraphics[scale=0.6]{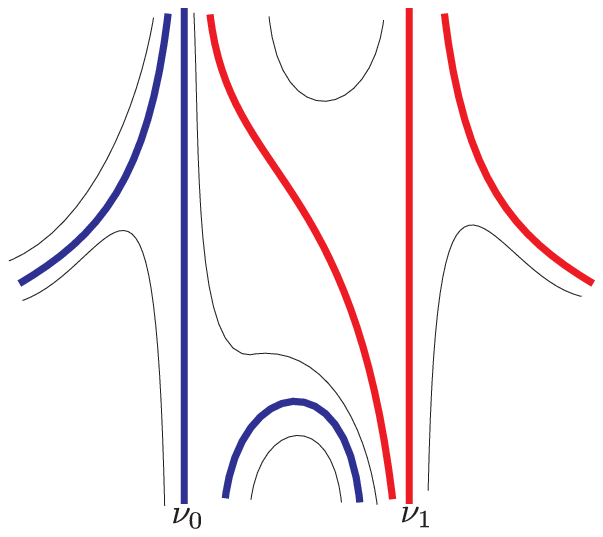}}
\end{center}
\caption{\small Separatrix configurations of $\mathscr{F}(p)$ and of $\mathscr{F}(q)$. 
In red there are connected components of level $0$ of $p$ and $q$. In blue, of level $1$.} 
\label{ultima}
\end{figure}
Analogously, for the function $q$, after analyzing the signal of $N_1(x)$ near $0$ and in the interval $(0,1)$, and also the signal of $N_0(x)$ in the interval $(0,1)$ and near $1$, we obtain the separatrix configuration of $\mathscr{F}(q)$ as depicted in (b) of Fig. \ref{ultima}. 

We claim that there are no \emph{increasing} bijection $\sigma:B(p) \to B(q)$ and no isomorphism nor anti-isomorphism $g$ between $S\mathfrak{S}_{\mathscr{F}(p)}$ and $S\mathfrak{S}_{\mathscr{F}(q)}$ such that 
$\sigma (p(\nu_i)) = q (g(\nu_i))$, for $i = 0,1$. 
So, in particular, condition (b) of corollaries \ref{coro:main}, \ref{varcorollary:main} and \ref{cor5.1} will not be satisfied with an increasing $\sigma$. 

Indeed, if $\sigma$ is increasing, then $\sigma(0) = 0$ and $\sigma(1) = 1$. 
In particular, an existing $g$ should send $\nu_0$ onto $\nu_1$, because it must preserve separability of leaves). 
Also, it should send $\alpha$, the inseparable leaf to $\nu_0$ contained in $(\nu_0, \nu_1)$, onto the inseparable leaf (related to $\mathscr{F}(q)$) to $\nu_1$ in $(\nu_0, \nu_1)$, because it should also preserve inseparable leaves. 
Analogously, it should send $\beta$, the inseparable leaf to $\nu_1$ contained in $(\nu_0, \nu_1)$ onto the inseparable leaf to $\nu_0$ in $(\nu_0, \nu_1)$. 
But $|\nu_0,\alpha,\beta|^+$ and $g(\nu_0)|g(\alpha)|g(\beta)$, hence $g$ cannot be an isomorphism or anti-isomorphism. 
\end{example}

Next example proves that away from the linear-like class only the behaviour of the separatrices does not characterize the foliation. 
\begin{example}\label{676767} 
Let 
$$
f(x,y) = \left(x - \arctan y +\pi \right) (2 x + \pi) \left(x - \arctan y \right) e^{2 y}. 
$$
First we show that this is a submersion. 
For each $y$, the partial derivatives $f_x (x,y)$ and $f_y (x,y)$ are polynomials of degrees $2$ and $3$ in $x$, respectively. 
The second one can be written as $(2x+\pi) g(x,y)$, where $g(x,y)$ is a polynomial of degree $2$ in $x$ for each $y$. 
We have $f_x (-\pi/2,y) \neq 0$ for all $y$. 
Further, the resultant, in $x$, between $f_x (x,y)$ and $g(x,y)$ is given, up to a factor of a positive function of $y$, and after substituting $y = \tan z$, by
$$
\pi^2 + \cos^4 z - 4 (z + \cos^2 z)^2,  
$$
which is a strictly positive function in the interval $(-\pi/2, \pi/2)$, and so there does not exists $y \in \R$ such that $f_x (x,y)$ and $g(x,y)$ have a common zero in $x$. 
Hence $f$ is a submersion. 

Clearly the straight line $\nu_{-\pi/2}$ and the curves $\gamma_1 = \{ (\arctan y - \pi, y)\ |\ y\in \R\}$ and $\gamma_2 = \{ (\arctan y,y)\ |\ y\in \R\}$ correspond to the three connected components of the level set $L_0$. 
Moreover, it is easy to see that $f$ is negative in the region to the left of $\gamma_1$ and in the region $(\nu_{-\pi/2}, \gamma_2)$, and it is positive in the other two connected components of $\R^2\setminus L_0$. 
Further, $f$ assumes all the negative values in the former two regions as well it assumes all the positive values in the last two ones. 
Now for any given $c\in \R$, $c\neq 0$, we observe that for fixed $y$ the discriminant of the polynomial of degree $3$ $x\mapsto f(x,y) - c$ is positive (resp. negative) for $y>0$ (resp. $y<0$) with $|y|$ big enough. 
It follows that for $y>0$ (resp. $y<0$) with $|y|$ big enough, the equation $f(x,y) = c$ has three (resp. one) real solutions. 
So, following the ideas of the proof of \cite[Lemma 3.5.]{BS} (with the signal of the discriminant being the opposite of $D$), we get that for any $c\neq 0$, the level set $f^{-1}(c)$ has two connected components. 
In particular, in the regions $(\gamma_1,\nu_{-\pi/2})$ and $(\nu_{-\pi/2}, \gamma_2)$ there is exactly one connected components of $L_c$ for each $c>0$ and for each $c<0$, respectively. 
Hence acting analogously as in the proof of Proposition \ref{prop-principal}, it follows that $\gamma_1$ is inseparable to $\nu_{-\pi/2}$, as well as $\nu_{-\pi/2}$ is inseparable to $\gamma_2$. 

Since, as observed right above, for any $c\neq 0$ the equation $f(x,y) = c$ for fixed $y$ has only one real solution if $y$ is negative enough, it follows that the ordinary leaves in $(\gamma_1, \nu_{-1/2})$ and in $(\nu_{-1/2}, \gamma_2)$ have its ends in the region $y>0$. 
Hence the separatrix configuration of $\mathscr{F}(f)$ with level sets is as shown in Fig. \ref{fig:exp-1}. 
\begin{figure}[h]
\begin{center}
\includegraphics[scale=0.6]{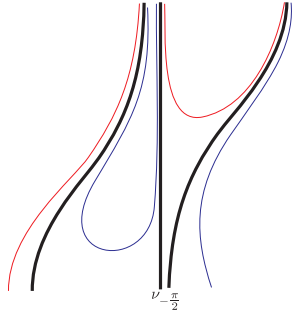}
\end{center}
\caption{\small Separatrix configuration of $\mathscr{F}(f)$. 
The inseparable leaves are drawn in black. 
In blue (resp. red), we draw ordinary leaves where $f$ assumes positive (resp. negative) values.} 
\label{fig:exp-1}
\end{figure} 
In particular, in the region $(\gamma_1, \nu_{-1/2})$, the behavior of the ordinary leaves are contrary to what it should be if we were in the linear-like case, because it does not respect Lemma \ref{prop-comportamento}. 
\end{example} 

Next example justifies why we ask that both $\ell$ and $h$ preserve orientation in our definition of o-topological equivalence of functions. 
\begin{example}
Let $p$ and $q$ be defined by 
$$
p(x,y) = x + (x+1)^2 x^2 (x-1)^2 y, \ \ \ \ q(x,y) = p(-x, y). 
$$
Following the ideas of Section \ref{section:linear} as above, it is simple to see that $p$ and $q$ are submersions such that the separatrix configurations of ${\mathscr{F}(p)}$ and of ${\mathscr{F}(q)}$ with level sets are as given in Fig. \ref{Ingridfara}. 
\begin{figure}[!h]
\begin{center}
\subfigure[$S\mathfrak{S}_{\mathscr{F}(p)}$]{\includegraphics[scale=0.6]{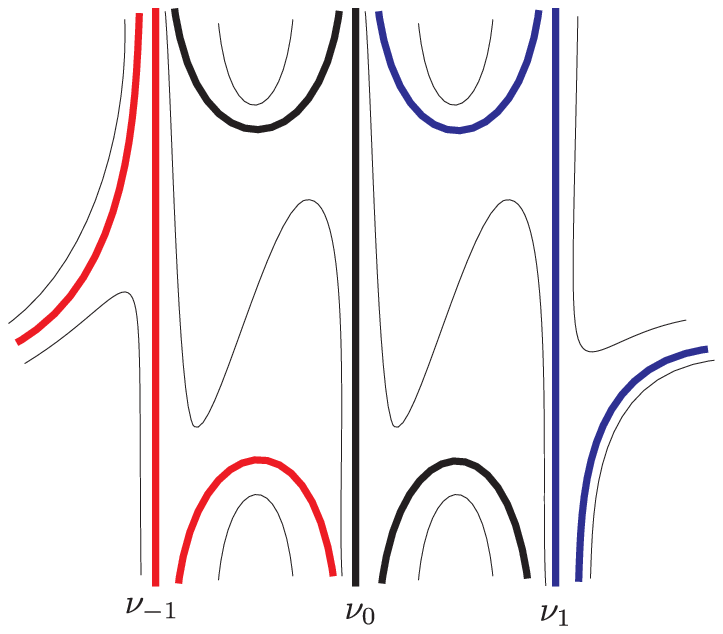}}\ \ \ \ \ \ \ \ \ \ \
\subfigure[$S\mathfrak{S}_{\mathscr{F}(q)}$]{\includegraphics[scale=0.6]{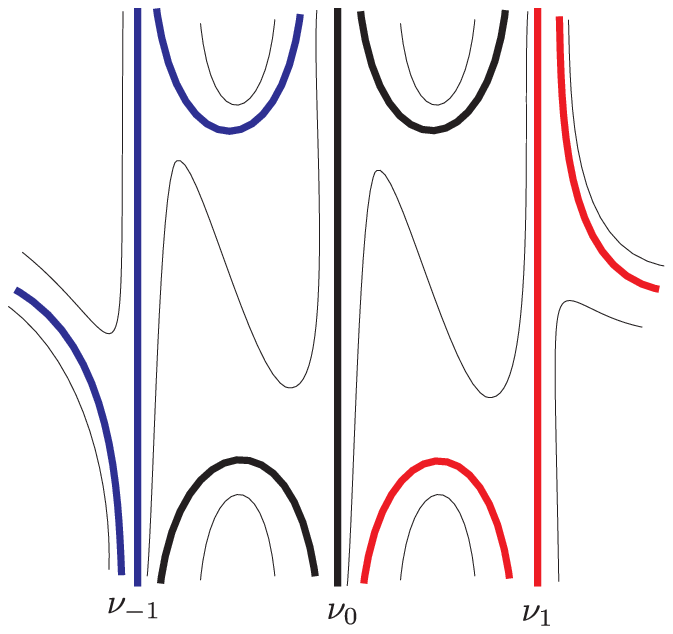}}
\end{center}
\caption{\small Separatrix configurations of $\mathscr{F}(p)$ and of $\mathscr{F}(q)$. 
Inseparable leaves where the functions assume the value $0$ are painted in black. 
In blue the inseparable leaves where the functions assume $1$ and in red, the ones where the functions assume $-1$.} 
\label{Ingridfara}
\end{figure} 

By taking $\ell(t) = t$ and $h(x,y) = (-x, y)$, we have 
$$
\ell\circ q = p\circ h. 
$$ 
Although, it is not difficult to verify that we can not define an isomorphism between $\mathfrak{S}_{\mathscr{F}(p)}$ and $\mathfrak{S}_{\mathscr{F}(q)}$: because if $g$ is such an isomorphism, we must have $g(\nu_0) = \nu_0$, also $g(\nu_{-1}) = \nu_{-1}$ or $g(\nu_{-1}) = \nu_{1}$; but for any of these two possibilities, it follows that cycles will have their orientation reverted. 
Therefore, the foliations $\mathscr{F}(p)$ and $\mathscr{F}(q)$ are not o-topologically equivalent. 

Since in our paper we are relating topological equivalence of functions with topological equivalence of the foliations given by them, it would be awkward to have $p$ and $q$ o-topologically equivalent together with $\mathscr{F}(p)$ and $\mathscr{F}(q)$ \emph{not} o-topologically equivalent. 
It is because of this that we ask that both $\ell$ and $h$ must preserve orientation in our definition. 
 \end{example}

Next example shows that all the possibilities presenting in configurations \ref{wa}, \ref{wb} and \ref{wc} can actually happen. 

\begin{example} 
For the cases in configuration \ref{wa}, all the examples above provide explicit linear-like submersions. 

For cases (a) and (d) of configuration \ref{wb}, the third and the second strips in (a) of Example \ref{degree2} given by $p$ provide explicit submersions. 
For case (c), let again $p$ of this example and define $q(x,y) = p(-x, y)$. 
The second strip of $\mathscr{F}(q)$ will provide an example. 
Finally, for case (b), let 
$$
p(x,y) = x (2- x^2) + (x-1)^2 (x+1)^2 y. 
$$ 
As above, it is simple to see that $p$ is a submersion such that the strip $[-1,1]\times \R$ has the same separatrix configuration as given in (b) of configuration \ref{wb}. 

For (a) and (b) of configuration \ref{wc}, we consider the polynomials  
$$
\begin{aligned}
p(x,y) & = x (x-1) + x^2(x-1)^2y,\\
q(x,y) & = x (1 - 2 x) (x-1) + x^2(x-1)^2y, 
\end{aligned}
$$ 
respectively. 
As above, it is simple to see that they are submersions and that their separatrix configuration in $[0,1]\times \R$ agree with the ones in (a) and (b) of configuration\ref{wc}, respectively. 
\end{example}

The following example exhibits linear-like submersions $p$ such that the set $Z_p$ is not discrete. 

\begin{example} 
Let $u: [0,\infty) \to \R$ be the $C^{\infty}$ function defined by $u(0) = 0$ and 
$$
u(x) = e^{-\frac{1}{x}}\sin^2\frac{\pi}{x}, \ x\neq 0, 
$$ 
and let $s_i: \R \to \R$, $i=1,2,3$, be the $C^{\infty}$ functions given by 
$$
s_1(x) =\left\{ \begin{array}{ll} 
u(x), & x \geq 0,\\
e^{\frac{1}{x}}, & x<0, 
\end{array}\right. \ \ \ \ 
s_2(x) = \left\{ \begin{array}{ll} 
	u(x), & x\geq0,\\
	u(-x), & x<0, 
\end{array}\right. \ \ \ \ 
s_3(x) = \left\{ \begin{array}{ll} 
	u(x), & x\geq 0, \\
	0, & x < 0. 
\end{array}\right. 
$$ 
We define 
$$
f_i(x,y) = x + s_i(x) y, \ i=1,2,3. 
$$ 
Denoting $\Z^*=\Z\setminus\{0\}$ and $\Z^+ = \Z\cap\{x\ |\ x>0\}$, we have 
$$
Z_{f_1} = \left\{\frac{1}{n}\ |\ n\in \Z^+ \right\}\cup \{0\}, \ \ \ \ Z_{f_2} = \left\{\frac{1}{n}\ |\ n\in \Z^* \right\}\cup\{0\}, \ \ \ \ Z_{f_3} = Z_{f_1}\cup\{x\ |\ x\leq 0\}. 
$$
In what follows we denote $1/0 = \infty$. 
By Lemma \ref{submersion} it follows that $f_i$, $i=1,2,3$, are submersions. 
Moreover, by Proposition \ref{prop-principal}, the straight lines $\nu_{1/n}$, $n\in Z^+$, together with the connected components of $L_{1/n}$ contained in $(1/(n+1),1/n)\cup(1/n, 1/(n-1))$, $n\in Z^+$, are inseparable leaves of $\mathscr{F}(f_i)$, $i=1,2,3$. 
In the case of $f_2$, all the straight lines $\nu_{-1/n}$, $n\in \Z^+$ are also inseparable leaves with the connected components of $L_{-1/n}$ contained in $(-1/(n-1),-1/n)\cup(-1/n, -1/(n+1))$, $n\in Z^+$. 

By applying Corollary \ref{92345}, it follows that $\nu_0$ is a limit separatrix in cases of $f_2$ and $f_3$. 
In case of $f_1$, $\nu_0$ is inseparable with a connected component of $L_0$ contained in $(-\infty, 0)$. 
In the case of $f_2$ there are inseparable leaves converging to $\nu_0$ from both ``sides'' of it. 
In the cases of $f_1$ and $f_3$, there are inseparable leaves converging from just one side. 
Further, in case of $f_3$, we observe that the foliation is trivial in the region $x<0$: the leaves are the straight lines $\nu_x$, $x<0$. 
Finally, we observe that in this last case, the limit-separatrix $\nu_0$ is the only leaf in the boundary of a canonical region, this does not happen in the other cases. 

To complete the separatrix configurations of $\mathscr{F}(f_i)$, $i=1,2,3$, we apply the results of Section \ref{section:linear}, obtaining that they are as depicted in Fig. \ref{ex1}. 
\begin{figure}[h!]
\begin{center}
\subfigure[$S\mathfrak{S}_{\mathscr{F}(f_1)}$]{\includegraphics[scale=0.55]{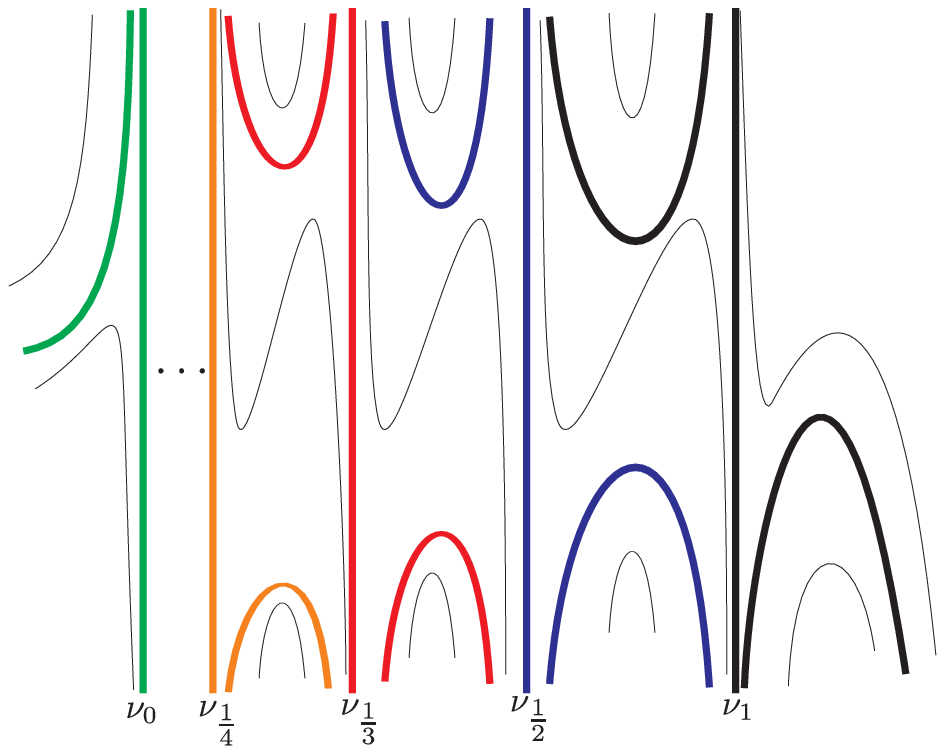}} \ \ \ \ \ \ \ \
\subfigure[$S\mathfrak{S}_{\mathscr{F}(f_2)}$]{\includegraphics[scale=0.55]{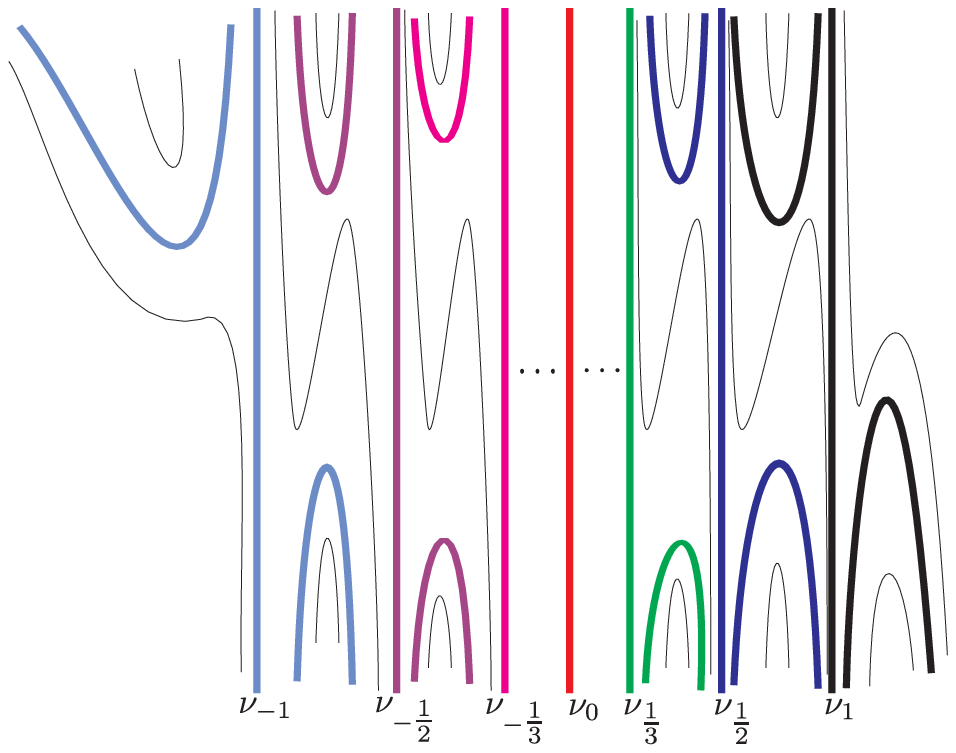}}
\subfigure[$S\mathfrak{S}_{\mathscr{F}(f_3)}$]{\includegraphics[scale=0.55]{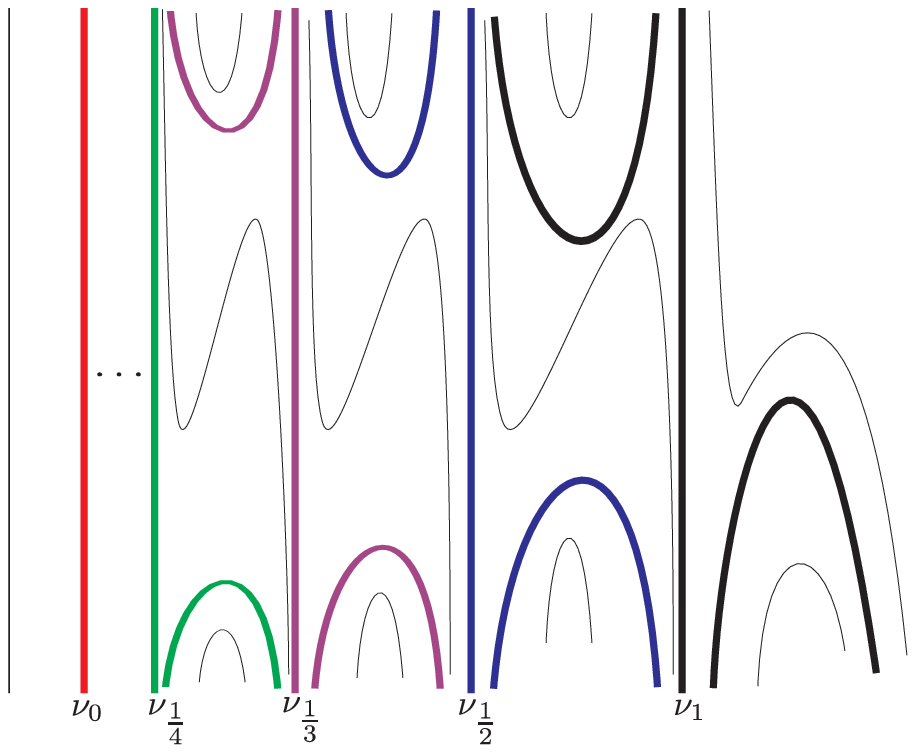}}
\end{center}
\caption{\small Separatrix configurations of the foliations $\mathscr{F}(f_i)$, $i=1,2,3$. 
Inseparable leaves are painted with the same color. 
In (a) there is exactly one curve inseparable with the straight line $\nu_0$ in the region $x<0$. 
In (b) and (c), $\nu_0$ is a limit separatrix. 
In (c), the region $x<0$ is trivial.}\label{ex1}
\end{figure} 
\end{example}

The last example provides a linear like submersion $f$ for any given closed subset $Z\subset \R$, such that $Z_f$ matches $Z$. 
\begin{example} 
Given $Z$ a closed set in $\R$, we know that $\R\setminus Z$ is the disjoint union of an enumerable quantity of disjoint open intervals of the form $(a_n, b_n)$, where one of the $a_n$ might be $-\infty$ and one of the $b_n$ might be $\infty$. 
In case $(a_n, b_n)$ is bounded, we define the $C^{\infty}$ function $s_n(x) = e^{1/\big((x-a_n) (x-b_n)\big)}$ for $x \in (a_n, b_n)$ and $s_n(x) = 0$ for $x$ away from $(a_n, b_n)$. 
For an interval of the form $(-\infty, b_n)$ we put $s_n(x) = e^{1/(x - b_n)}$ for $x\in (-\infty, b_n)$ and $s_n(x) = 0$ when $x$ is not in $(-\infty, b_n)$. 
Analogously, when $b_n = \infty$, we put $s_n(s) = e^{1/(a_n - x)}$ in $(a_n, b_n)$ and $0$ away from this interval. 
Then we consider $s(x) = \sum_{n=0}^{\infty} s_n(x)$. 
This is a $C^{\infty}$ function with zero set being exactly $Z$. 

By defining $f(x,y) = -x + s(x) y$ we will get a $C^{\infty}$ submersion by Lemma \ref{submersion}, because $s$ is flat in each $a_n$ and $b_n$. 
In each strip $[a_n, b_n]\times \R$ the foliation $\mathscr{F}(f)$ has separatrix configuration as (d) of configuration \ref{wc} if $[a_n, b_n]$ is bounded and as (a) (with $+$ and $-$ exchanged) and (b) of configuration \ref{wa} if $a_n = -\infty$ or $b_n = \infty$, respectively. 

By Corollary \ref{92345} it follows that for any $z\in \overline{\cup_i\{a_i, b_i\}}\setminus \cup_{i}\{a_i, b_i\}$, the straight line $\nu_z$ is a limit separatrix of $\mathscr{F}(f)$. 
Then we choose an $x$ in each connected component of the complement of $\overline{\cup_i(a_i, b_i)}\setminus \cup_{i}(a_i, b_i)$, and the curve $\nu_x$ will be an ordinary leaf completing the separatrix configuration of $\mathscr{F}(f)$. 
\end{example}

\section{Relation between topological equivalence of finite linear-like submersions and topological equivalence of the foliations given by them}\label{section:final}

As discussed in the introduction section, the topological (resp. o-topological) equivalence of submersion functions implies the topological (resp. o-topological) equivalence of the foliations given by them. 
The reciprocal result is not true even in the finite linear-like case, as examples in Section \ref{section:examples} show. 
In this section, we add conditions in order to obtain a reciprocal result in the finite linear-like class. 
The most general results we obtain is Theorem \ref{mmain} and its corollary. 
The results announced in the introduction section will be consequences of them, and are proved below. 

We need the following general lemma on regular foliations: 
\begin{lemma}\label{canonicas}
Let $\mathcal{A}$ be a canonical region of a regular foliation $\mathscr{F}$. 
The following properties hold for leaves $\gamma_1, \gamma_2, \alpha, \beta \in \mathscr{F}$ such that $\gamma_1, \gamma_2$ are contained in $\mathcal{A}$ and $\alpha, \beta$ are not contained in $\mathcal{A}$: 
\begin{enumerate}[label={{\textnormal{(\alph*)}}}]
\item\label{tyu} $\gamma_1|\gamma_2|\alpha$ or $\gamma_2|\gamma_1|\alpha$. 
\item\label{tyu2} The chordal relations of the triples $\gamma_1$, $\alpha$, $\beta$ and $\gamma_2$, $\alpha$, $\beta$ are the same. 
\end{enumerate} 
\end{lemma}
\begin{proof}
If \ref{tyu} is not true, then $|\gamma_1, \gamma_2, \alpha|^{\pm}$ and hence, in particular, $\alpha$ is contained in $(\gamma_1, \gamma_2)$, which is an open set contained in $\mathcal{A}$, a contradiction. 

Now we prove \ref{tyu2}. 
By using Section 2.3 of Kaplan's paper \cite{K}, it is possible to deliver a very technical proof of this statement by applying the axioms and results on abstract chordal systems. 
We prefer to give a more self-contained proof, by using only our definitions in Section \ref{890jd}. 
By \ref{tyu} we can assume 
\begin{equation}\label{assump}
\gamma_1 | \gamma_2 | \alpha.
\end{equation} 
We have the following possibilities for $\gamma_1$, $\alpha$ e $\beta$: (i) $\gamma_1 | \alpha | \beta$, (ii) $\gamma_1 | \beta | \alpha$, (iii) $\alpha | \gamma_1 | \beta$ and (iv) $|\gamma_1, \alpha, \beta|^{\pm}$. 
Clearly \eqref{assump} and (i) is equivalent to \eqref{assump} and $\gamma_2 | \alpha | \beta$. 
In case (ii), since $\beta$ cannot be contained in $[\gamma_1, \gamma_2]$ it follows that $\beta$ is contained in $(\gamma_2, \alpha)$. 
If $\beta$ does not separate $\gamma_2$ and $\alpha$, we let $B$ be the open connected set whose boundary is $\gamma_2 \cup \beta \cup \alpha$. 
The open connected region given by $(\gamma_1, \gamma_2]\cup B$ have as boundary $\gamma_1 \cup \beta \cup \alpha$, a contradiction with the assumption $\gamma_1 | \beta | \alpha$ of this case. 
In case (iii), since nor $\alpha$ nor $\beta$ is contained in $[\gamma_1, \gamma_2]$, it follows that $\alpha$ and $\beta$ are in distinct connected components of $[\gamma_1, \gamma_2]^c$, hence in particular they are in distinct connected components of $\gamma_2^c$, i.e. $\alpha | \gamma_2 | \beta$. 
Finally, in case (iv), since nor $\alpha$ nor $\beta$ is contained in $[\gamma_1, \gamma_2]$, it follows by \eqref{assump} that both $\alpha$ and $\beta$ are contained in the connected component of $\gamma_2^c$ not containing $\gamma_1$. 
In particular $\alpha$ does not separate $\gamma_2$ and $\beta$ as well as $\gamma_2$ does not separate $\alpha$ and $\beta$. 
If $\beta$ separates $\gamma_2$ and $\alpha$, it follows in particular that the region $(\gamma_1, \gamma_2] \cup [\gamma_2, \beta)$ is an open connected region not containing $\alpha$, a contradiction with our assumptions in this case. 
So $\gamma_2, \alpha, \beta$ form a cyclic triple. 
If $|\gamma_1, \alpha, \beta|^+$, there exist $p_1\in \gamma_1$, $p_2\in \alpha$ and $p_3 \in \beta$ such that $p_1p_2p_3p_1$ is a positive Jordan curve as in the definition of $|\gamma_1, \alpha, \beta|^+$. 
In particular, this curve cuts the leaf $\gamma_2$ twice in distinct points $q_1, q_2$. 
By sliding $q_2$ over $\gamma_2$ until it meets $q_1$ we obtain a positively oriented Jordan curve $q_1p_2p_3q_1$ with the properties guaranteeing that $|\gamma_2, \alpha, \beta|^+$. 
Similarly $|\gamma_1, \alpha, \beta|^-$ implies $|\gamma_2, \alpha, \beta|^-$. 
\end{proof}

\begin{theorem}\label{mmain}
Let $p_i(x,y) = r_i(x) + s_i(x) y$, $i=1,2$, be two finite linear-like submersions. 
Then $p_1$ is o-topologically equivalent to $p_2$ if and only if $Z_{p_1} = Z_{p_2} = \emptyset$ or each of the following conditions hold:
\begin{enumerate}[label={\textnormal{(\alph*)}}]
\item\label{cc1} $\mathscr{F}(p_1)$ and $\mathscr{F}(p_2)$ are o-topologically equivalent. 
\item\label{cc2} There exists an increasing bijection $\sigma: r_1(Z_{p_1})\to r_2(Z_{p_2})$ such that 
$$
\sigma(p_1(\nu_a)) = p_2(h(\nu_a)) 
$$ 
for each $a \in Z_{p_1}$, where  $h$ is an o-equivalence homeomorphism between $\mathscr{F}(p_1)$ and $\mathscr{F}(p_2)$. 
If $r_1(Z_{p_1})$ is a singleton, there further exist $m \notin r_1(Z_{p_1})$, an extension of $\sigma$ to an increasing bijection $\overline{\sigma}: r_1(Z_{p_1})\cup \{m\} \to r_2(Z_{p_2})\cup \{\overline{\sigma}(m)\}$ and a connected component $\xi$ of $p_1^{-1}(m)$ such that $\overline{\sigma}(m) = p_2(h(\xi))$. 
\end{enumerate}
\end{theorem} 
\begin{proof}
(Necessity) 
Assuming the existence of preserving orientation homeomorphisms $\ell: \R \to \R$ and $h: \R^2\to \R^2$ such that $\ell \circ p_1 = p_2 \circ h$, it follows that $h$ is an o-equivalence homeomorphism between $\mathscr{F}(p_1)$ and $\mathscr{F}(p_2)$, proving \ref{cc1}. 
Moreover, by Lemma \ref{prop-ida-1} and Theorem \ref{bifurcation}, it follows that $\sigma = \ell|_{r_1(Z_{p_1})}$ (resp. $\overline{\sigma} = \ell|_{r_1(Z_{p_1})\cup \{m\}}$ for any $m \notin {r_1(Z_{p_1})}$) defines an increasing bijection between $r_1\left(Z_{p_1}\right)$ and $r_2\left(Z_{p_2}\right)$ (resp. between $r_1\left(Z_{p_1}\right)\cup \{m\}$ and $r_2\left(Z_{p_2}\right)\cup\{\overline{\sigma}(m)\}$), and so the formula of \ref{cc2} follows from the identity $\ell \circ p_1 = p_2 \circ h$.

(Sufficiency) 
If $Z_{p_1} = Z_{p_2} = \emptyset$, we let $\tau_i$ be the constant signal of $s_i(x)$, and $h_i: \R^2 \to \R^2$ be the orientation preserving homeomorphism defined by $h_i(x,y) = \big(x, \tau_i r_i(x) + \tau_i s_i(x) y\big)$, with inverse $h_i^{-1}(x,y) = \big(x, (y - \tau_i r_i(x))/(\tau_i s_i(x))\big)$, and observe that $p_i = \tau_i q\circ h_i$, $i=1,2$, where $q$ is the projection $q(x,y) = y$. 
Hence $p_i$ is o-topologically equivalent to the function $\tau_i q$, $i=1,2$. 
As $q$ and $-q$ are o-topologically equivalent it follows that $p_1$ and $p_2$ are o-topologically equivalent. 
(We can as well invoke Sharko and Soroka \cite{sharko-15} already mentioned result to guarantee that $p_1$ and $p_2$ are both o-topologically equivalent to a projection.) 

So we assume $Z_{p_i}\neq \emptyset$, $i=1,2$ (we can not have just one of them empty, because of condition \ref{cc2} and Theorem \ref{bifurcation}) and write 
$$
Z_{p_i} = \{a_1^i, a_2^i, \ldots, a_{k_i}^i\}, 
$$ 
with $a_{j}^i < a_{j+1}^i$, $j=1,\ldots,k_i-1$, $i = 1,2$. 
We first notice that being an equivalence homeomorphism between $\mathscr{F}(p_1)$ and $\mathscr{F}(p_2)$, $h$ must carry the region $(-\infty, a_1^1]\times \R$ onto $(-\infty, a_1^2]\times \R$ or onto $[a_{k_2}^2,\infty) \times \R$, in the light of configurations \ref{wa}, \ref{wb} and \ref{wc} of Section \ref{section:linear}. 
So up to considering the function $\widetilde{p}_2(x,y) = p_2(-x,-y)$, which is o-topologically equivalent to $p_2$, in the place of $p_2$, we can assume that $h$ carries $(-\infty, a_1^1]\times \R$ onto $(-\infty, a_1^2]\times \R$. 
In particular $h(\nu_{a_1^1}) = \nu_{a_1^2}$. 
In order to organize the remaining of the proof, we divide it in 12 steps. 

\emph{Step 1:} \emph{We claim that $k_1 = k_2$, i.e., $Z_{p_1}$ and $Z_{p_2}$ have the same cardinality, and that $h(\nu_{a_j^1}) = \nu_{a_j^2}$, $j=1,\ldots, k$, with $k=k_1=k_2$. 
In particular, $h$ sends the region $(\nu_{a_j^1}, \nu_{a_{j+1}^1})$ onto  $(\nu_{a_j^2}, \nu_{a_{j+1}^2})$, $j=0,\ldots, k$, where $a_0^i = -\infty$ and $a_k^i = \infty$, $i=1,2$}. 

If $k_1=1$, then the open region $(a_1^2, \infty)\times \R = h\left((a_1^1, \infty) \times \R\right)$ has exactly one separatrix, and so $k_2 = k_1 = 1$, proving Step 1 in this case. 
So we assume that $k_1,k_2 >1$ and that $h(\nu_{a_l^1}) = \nu_{a_l^2}$ for $l=1,\ldots, j < k_1$. 
It is enough to prove that $h(\nu_{a_{j+1}^1}) = \nu_{a_{j+1}^2}$. 
Indeed, the closed region $[\nu_{a_j^1}, \nu_{a_{j+1}^1}]$ is mapped by $h$ onto the closed region $[\nu_{a_j^2}, h(\nu_{a_{j+1}^1})]$, contained in $\{x\ |\ x\geq a_j^2\} \times \R$. 
According to the configurations \ref{wc} and \ref{wb} studied in Section \ref{section:linear}, besides $\nu_{a_j^1}$ and $\nu_{a_{j+1}^1}$, the region $[\nu_{a_j^1}, \nu_{a_{j+1}^1}]$ contains either one or two more separatrices of $\mathscr{F}(p_1)$ depending whether $\nu_{a_j^1}$ and $\nu_{a_{j+1}^i}$ are or not in the same level set of $p_1$, respectively. 
In the first case, being $\alpha$ the existing extra separatrix, it follows by the configuration \ref{wc} that $\alpha$, $\nu_{a_j^1}$ and $\nu_{a_{j+1}^1}$ are in the same level set of $p_1$. 
So by assumption \ref{cc2} it follows that $\nu_{a_j^2}$, $h(\alpha)$ and $h(\nu_{a_{j+1}^1})$ contained in $[\nu_{a_j^2}, h(\nu_{a_{j+1}^1})]$ are in the same level set of $p_2$. 
Since these three leaves are the only separatrices of $\mathscr{F}(p_2)$ contained in this set, we can conclude from a glance in configurations \ref{wa}, \ref{wb} and \ref{wc} that we must have $h(\nu_{a_{j+1}^1}) = \nu_{a_{j+1}^2}$. 
Analogously, assuming the second case, we have from configuration \ref{wb} that one of the separatrices in $(\nu_{a_j^1}, \nu_{a_{j+1}^1})$, say $\alpha$, is such that $p_1(\nu_{a_j^1}) = p_1(\alpha)$, and the other, say $\beta$, satisfies $p_1(\beta) = p_1 (\nu_{a_{j+1}^1})$, with $p_1(\alpha) \neq p_1(\beta)$. 
Hence by assumption \ref{cc2}, it follows that 
\begin{equation}\label{meio}
p_2(\nu_{a_{j}^2}) = p_2\left(h(\alpha)\right) \neq p_2\left(h(\beta)\right) = p_2(h(\nu_{a_{j+1}^1})). 
\end{equation} 
Since these four separatrices of $\mathscr{F}(p_2)$ are the only ones in the region $[\nu_{a_{j}^2}, h(\nu_{a_{j+1}^1})]$, and they satisfy \eqref{meio}, it follows again by a glance in the configurations \ref{wa}, \ref{wb} and \ref{wc} that we must have $h(\nu_{a_{j+1}^1}) = \nu_{a_{j+1}^2}$. 
So Step 1 is proved. 

\emph{Step 2:} \emph{Assumption \ref{cc2} is also valid throughout $\mathfrak{S}_{\mathscr{F}(p_1)}$, i.e., $\sigma \left(p_1(\gamma)\right) = p_2 \left(h(\gamma)\right)$ for any separatrix $\gamma$ of $\mathscr{F}(p_1)$}. 
Indeed, recall that two leaves of $\mathscr{F}(p_i)$ which are inseparable to each other are in the same level set of $p_i$. 
Further, from Proposition \ref{prop-principal}, any separatrix of $\mathscr{F}(p_i)$ is inseparable to a separatrix of the form $\nu_{a_j^i}$. 
Therefore, since $h$ sends pairs of inseparabel leaves of $\mathscr{F}(p_1)$ onto pairs of inseparable leaves of $\mathscr{F}(p_2)$, Step 2 follows by assumption \ref{cc2}. 

\emph{Step 3:} We define $\ell: \R \to \R$ to be any homeomorphism such that $\ell|_{r_1(Z_{p_1})} = \sigma$ or $\ell|_{r_1(Z_{p_1})\cup \{m\}} = \overline{\sigma}$, in case $r_1(Z_{p_1})$ is a singleton. 
Since $\sigma$ (or $\overline{\sigma}$) is increasing, it follows that $\ell$ is orientation-preserving. 

\emph{Step 4:} In order to finish the proof, it is enough to construct an isomorphism $g: \mathscr{F}(p_1) \to \mathscr{F}(p_2)$ such that $\ell \left(p_1 (\gamma)\right) = p_2\left(g(\gamma) \right)$ for each $\gamma \in \mathscr{F}(p_1)$. 
Because with this isomorphism in hands it follows by Theorem \ref{KK} that there exists an o-equivalence homeomorphism $H: \R^2 \to \R^2$ between ${\mathscr{F}(p_1)}$ and ${\mathscr{F}(p_2)}$ such that $H(\gamma) = g(\gamma)$ for each $\gamma \in \mathscr{F}(p_1)$. 
So for any $x\in \R^2$, by denoting $\gamma_x$ the leaf of $\mathscr{F}(p_1)$ containing $x$, it follows that 
$$
\ell (p_1(x)) = \ell (p_1(\gamma_x)) = p_2(g(\gamma_x)) = p_2(H(\gamma_x)) = p_2(H(x)), 
$$
hence $p_1$ and $p_2$ are o-topologically equivalent. 

In the following steps, we will construct the isomorphism $g$. 

\emph{Step 5:} We first define 
$$
g(\gamma) = h(\gamma)
$$
for any $\gamma$ in the set of separatrices of $\mathscr{F}(p_1)$, and also for $\gamma = \xi$ when $r_1\left(Z_{p_1}\right)$ is a singleton. 
By Step 2, it follows that $\ell (p_1(\gamma)) = p_2(g(\gamma))$. 

\emph{Step 6:} Here for each $j = 0, \ldots, k$, we define a bijection $g_j$ from the set of leaves $\mathscr{F}(p_1)$ foliating the strip region $(a_{j}^1, a_{j+1}^1) \times \R$ onto the set of leaves of $\mathscr{F}(p_2)$ foliating $(a_{j}^2, a_{j+1}^2) \times \R$ (see Step 1) as follows: 

We let $\alpha^1(t)$ be a transversal curve in $(a_{j}^1, a_{j+1}^1) \times \R$ parametrized by the level sets of $p_1$, i.e., such that $p_1(\alpha^1(t)) = t$ (for instance, for any $x_0 \in (a_{j}^1, a_{j+1}^1)$, we can take $\alpha^1(t) = \big(x_0, N_t(x_0)\big)$, where $N_t(x)$ is defined in Lemma \ref{submersion} for function $p_1$). 
We define $\alpha^2(t)$ having the same meaning in $(a_{j}^2, a_{j+1}^2) \times \R$ for function $p_2$. 

Each leaf $\gamma^i$ of $\mathscr{F}(p_i)$ foliating $(a_{j}^j, a_{j+1}^i) \times \R$ can be uniquely denoted by $\gamma^i_t$ meaning the only leaf of $\mathscr{F}(p_i)$ crossing the transversal $\alpha^i$ in the point $\alpha^i(t)$. 
So we define 
$$
g_j(\gamma^1_t) = \gamma^2_{\ell(t)}, 
$$
for each $t \in \R$. 
This is a bijection between the leaves of $\mathscr{F}(p_1)$ foliating $(a_{j}^1, a_{j+1}^1) \times \R$ onto the leaves of $\mathscr{F}(p_2)$ foliating $(a_{j}^2, a_{j+1}^2) \times \R$. 
By assumption \ref{cc2}, $g_j(\gamma) = h(\gamma)$ for any separatrix $\gamma$ in $(a_{j}^1, a_{j+1}^1) \times \R$, as well as if $\gamma$ is $\xi$ (if it is the case), hence $g_j$ agrees with $g$ defined in Step 5 in these cases. 
By construction, $g_j$ satisfies $\ell \circ p_1 = p_2\circ g_j$. 

\emph{Step 7:} We then define $g: \mathscr{F}(p_1) \to \mathscr{F}(p_2)$ by extending the bijection defined in Step 5 to the entire $\mathscr{F}(p_1)$ as 
$$
g(\gamma) = g_j(\gamma) 
$$ 
if $\gamma$ is contained in the strip $(a_{j}^1, a_{j+1}^1)\times \R$, $j=0, \ldots, k$, according to Step 6. 
By construction it follows that $g$ is a well defined bijection and also that it satisfies $\ell\circ p_1 = p_2\circ g$. 

It remains to prove that $g$ is an isomorphism. 
I.e., given $\gamma_1, \gamma_2, \gamma_3 \in \mathscr{F}(p_1)$, we have to show that the chordal relations of them are the same than the chordal relations of $g(\gamma_1), g(\gamma_2), g(\gamma_3) \in \mathscr{F}(p_2)$. 

\emph{Step 8:} If $\gamma_1 = \gamma_{t_1}, \gamma_2 = \gamma_{t_2}, \gamma_3 = \gamma_{t_3}$ are contained in the same strip region $(a_{j}^1, a_{j+1}^1) \times \R$ with $t_1<t_2<t_3$, we have $\gamma_{t_1}|\gamma_{t_2}|\gamma_{t_3}$. 
The definition of $g$ in Step 7 proves that $g(\gamma_{t_1}) | g(\gamma_{t_2}) | g(\gamma_{t_3})$, and so $\gamma_1$, $\gamma_2$ and $\gamma_3$ have the same chordal relations than $g(\gamma_1)$, $g(\gamma_2)$ and $g(\gamma_3)$. 
Also, since $g=h$ in the set of separatrices, and $h$ is an isomorphism, the same happens if $\gamma_1$, $\gamma_2$ and $\gamma_3$ are separatrices. 
Hence \emph{we can assume that $\gamma_1$, $\gamma_2$ and $\gamma_3$ are not in the same strip region and that one of them is not a separatrix}. 

Before to proceed, we need some properties of $h$ in the next three steps. 
With them in hands, we will be able to prove that $g$ preserves chordal relations by using the fact that $h$ preserves chordal relations in Step 12. 

\emph{Step 9:} Since each level set of $p_i$ has exactly one connected component in $(a_{j}^i, a_{j+1}^i) \times \R$, it follows that \emph{$\gamma_1|\gamma_2|\gamma_3$ is equivalent to $p_i(\gamma_1) < p_i(\gamma_2) < p_i(\gamma_3)$ or to $p_i(\gamma_3) < p_i(\gamma_2) < p_i(\gamma_1)$ for leaves $\gamma_1,\gamma_2, \gamma_3 \in \mathscr{F}(p_i)$ in the same strip region $(a_{j}^i, a_{j+1}^i) \times \R$}. 

\emph{Step 10:} \emph{Given $\gamma_1, \gamma_2 \in \mathscr{F}(p_1)$ in the same strip region $(a_{j}^1, a_{j+1}^1) \times \R$ such that $p_1(\gamma_1) < p_1(\gamma_2)$, then $p_2(h(\gamma_1)) < p_2(h(\gamma_2))$}. 
Indeed, we first assume that in the strip region $(a_{j}^1, a_{j+1}^1) \times \R$ there are leaves $\alpha, \beta$ both of them separatrices or one separatrix and the other being $\xi$ (assumption \ref{cc2} guarantees the existence of at least one  strip region like this). 
So we have, say, $p_1(\alpha) < p_1(\beta)$, and it follows that $p_2(h(\alpha)) = \ell\circ p_1(\alpha) < \ell\circ p_1(\beta) = p_2(h(\beta))$, as $\ell$ is increasing. 
Therefore, since $h$ preserves chordal relations, it follows by Step 9 that for a given leaf $\gamma$ in this strip region: $p_1(\gamma)<p_1(\alpha)$, or $p_1(\alpha) < p_1(\gamma) <p_1(\beta)$ or $p_1(\beta) <p_1(\gamma)$ implies respectively that $p_2(h(\gamma)) < p_2(h(\alpha)) < p_2(h(\beta))$, or $p_2(h(\alpha)) < p_2(h(\gamma)) < p_2(h(\beta))$ or $p_2(h(\beta))<p_2(h(\gamma))$. 
So we conclude, using once more Step 9, that if $p_1(\gamma_1) < p_1(\gamma_2)$, then $p_2(h(\gamma_1)) < p_2(h(\gamma_2))$ for any leaves $\gamma_1, \gamma_2$ in this strip region. 
I.e., Step 10 is true in a strip region containing leaves $\alpha$ and $\beta$ as above. 

Now, in order to prove Step 10 for any strip region, it is enough to conclude that given a strip region containing $\alpha$ and $\beta$ as above, we are able to construct $\overline{\alpha}$ and $\overline{\beta}$ in both the \emph{adjacent} strip regions (when it is the case) such that $p_1(\overline{\alpha})<p_1(\overline{\beta})$ and $p_2(h(\overline{\alpha}) < p_2(\overline{\beta})$. 
Because then we apply the same reason than in the former case in order to prove that $p_1(\gamma_1)<p_1(\gamma_2)$ implies $p_2(h(\gamma_1)) < p_2(h(\gamma_2))$ for $\gamma_1, \gamma_2$ in this strip. 
So we proceed to construct $\overline{\alpha}$ and $\overline{\beta}$ in the region $(a_{j_0}^1, a_{j_0+1}^1)\times \R$ by assuming the existence of $\alpha$ and $\beta$ in the region $(a_{j_0-1}^1,a_{j_0}^1)\times \R$. 
The other situation is analogous. 
Centered at a point $b \in \nu_{a_{j_0}^1}$ there is an open ball $B$ such that 
$$
B^- = B\cap \left( (a_{j_0-1}^1, a_{j_0}^1)\times \R \right), \ \ \ \  B^+ = B \cap \left( (a_{j_0}^1, a_{j_0+1}^1)\times \R \right) 
$$ 
are contained in canonical regions of $\mathscr{F}(p_1)$. 
For any $a\in B^-$ and $c\in B^+$, it follows that the signal of $p_1(c)-p_1(a)$ is constant. 
We assume without loss of generality that $p_1(a) < p_1(b) <p_1(c)$ for every $a\in B^-$ and for every $c\in B^+$. 
We have $p_1(\alpha) = p_1(b)$ or $p_1(\beta) = p_1(b)$. 
We assume $p_1(\alpha) = p_1(b)$. 
So by taking $\gamma_a$ the leaf containing an $a\in B^-$ it follows that 
$p_1(\gamma_a) < p_1(\alpha)$ and so $p_2(h(\gamma_a)) < p_2(h(\alpha)) = p_2(h(b))$. 
With this we conclude that 
$$
p_2(h(a)) < p_2(h(b)) < p_2(h(c))
$$ 
for any $a\in B^-$ and $c \in B^+$, because analogously the signal of $p_2(h(c))-p_2(h(a))$ is constant. 
We fix $c_0\in B^+$ and take $\overline{\beta} = \gamma_{c_0}$, the leaf of $\mathscr{F}(p_1)$ containing $c_0$, and $\overline{\alpha}$ to be the separatrix contained in $(a_{j_0}^1, a_{j_0+1}^1)\times \R$ which is inseparable with $\gamma_{a_{j_0}^1}$. 
By construction we have $p_1(\overline{\alpha}) = p_1(b) < p_1(\overline{\beta})$ and $p_2(h(\overline{\alpha})) = p_2(h(b)) < p_2(h(c_0)) = p_2(h(\overline{\beta})$, finishing the proof of Step 10. 

\emph{Step 11:} \emph{For any ordinary leaf $\gamma \in \mathscr{F}(p_1)$, the leaves $h(\gamma)$ and $g(\gamma)$ are in the same canonical region}. 
Indeed, if $\alpha$ is a separatrix contained in the same strip region $(a_{j}^1, a_{j+1}^1) \times \R$ than $\gamma$, we have $p_1(\gamma) < p_1(\alpha)$ or $p_1(\gamma) > p_1(\alpha)$. 
From Step 10 it follows that $p_2(h(\gamma)) < p_2(h(\alpha))$ or $p_2(h(\gamma)) > p_2(h(\alpha))$, respectively. 
Since $p_2(g(\gamma)) < p_2 (g(\alpha))$ or $p_2(g(\gamma)) > p_2 (g(\alpha))$, respectively, we are done. 

\emph{Step 12:} \emph{The bijection $g$ is an isomorphism}. 
According to Step 8, it remains to analyze two possibilities for $\gamma_1, \gamma_2, \gamma_3$: (i) $\gamma_1$ is an ordinary leaf in a canonical region containing neither $\gamma_2$ nor $\gamma_3$ and $\gamma_2$ and $\gamma_3$ are not contained in the same canonical region, and (ii) $\gamma_1$ and $\gamma_2$ are ordinary leaves in the same canonical region not containing $\gamma_3$. 
For both cases, we will show that the chordal relations of the triples $g(\gamma_1), g(\gamma_2), g(\gamma_3)$ and $h(\gamma_1), h(\gamma_2), h(\gamma_3)$ are the same. 
This will conclude that $g$ is an isomorphism, because $h$ is an isomorphism. 

In case (i), it follows from Step 11 that $g(\gamma_1)$ and $h(\gamma_1)$ are in the same canonical region. 
Analogously, we have that $g(\gamma_i)$ and $h(\gamma_i)$ are in the same canonical region if $\gamma_i$ is an ordinary leaf, for $i=2,3$. 
Since $\gamma_2$ and $\gamma_3$ are not both ordinary leafs contained in the \emph{same} canonical region, it then follows by applying (possible more than once) Lemma \ref{canonicas} together with the fact that $g=h$ in the separatrices that $g(\gamma_1)$, $g(\gamma_2)$, $g(\gamma_3)$ have the same chordal relations than $h(\gamma_1)$, $h(\gamma_2)$, $h(\gamma_3)$. 

In case (ii), by \ref{tyu} of Lemma \ref{canonicas}, we can assume $\gamma_1|\gamma_2|\gamma_3$. 
It is enough to prove that $g(\gamma_1) | g(\gamma_2) | h(\gamma_3)$, because $h(\gamma_3) = g(\gamma_3)$ in case $\gamma_3$ is a separatrix, or $h(\gamma_3)$ and $g(\gamma_3)$ are in the same canonical region by Step 11 otherwise, and then \ref{tyu2} of Lemma \ref{canonicas} applies. 
We suppose on the contrary that this is not true. 
By \ref{tyu} of Lemma \ref{canonicas} again, we have $g(\gamma_2) | g(\gamma_1) | h(\gamma_3)$. 

Up to multiplying $p_1$ and $p_2$ by $-1$ we can assume that $p_1(\gamma_1)<p_1(\gamma_2)$, hence $p_2(g(\gamma_1))<p_2(g(\gamma_2))$ and, by Step 10, $p_2(h(\gamma_1))<p_2(h(\gamma_2))$. 
By analyzing the possible positions of $g(\gamma_i)$ and $h(\gamma_i)$, i.e., the signals of $p_2(g(\gamma_i))-p_2(h(\gamma_j))$, $i,j=1,2$, we see that in order to not contradict $h(\gamma_1)|h(\gamma_2)|h(\gamma_3)$ (because $h$ is an isomorphism) and $g(\gamma_2)|g(\gamma_1)|h(\gamma_3)$ we must have $p_2(h(\gamma_2))<p_2(g(\gamma_1))$. 
But then we conclude that $h(\gamma_3)$ is contained in $(h(\gamma_2), g(\gamma_1))$, which is an open connected set contained in the same canonical region containing $g(\gamma_i), h(\gamma_i)$, $i=1,2$, a contradiction with our assumption (ii). 
This finishes the proof. 
\end{proof}

\begin{corollary}\label{cor5.1}
Let $p_i(x,y) = r_i(x) + s_i(x) y$, $i=1,2$, be two finite linear-like submersions. 
Then $p_1$ is topologically equivalent to $p_2$ if and only if $Z_{p_1} = Z_{p_2} = \emptyset$ or each of the following conditions hold:
\begin{enumerate}[label={\textnormal{(\alph*)}}]
\item\label{ccc1} $\mathscr{F}(p_1)$ and $\mathscr{F}(p_2)$ are topologically equivalent. 
\item\label{ccc2} There exists a monotone bijection $\sigma: r_1(Z_{p_1})\to r_2(Z_{p_2})$ such that 
$$
\sigma(p_1(\nu_a)) = p_2(h(\nu_a)),\ \ \ \ a \in Z_{p_1}, 
$$
where $h$ is an equivalence homeomorphism between $\mathscr{F}(p_1)$ and $\mathscr{F}(p_2)$. 
\end{enumerate}
\end{corollary}
\begin{proof} 
(Necessity) If there are homeomorphisms $h: \R^2 \to \R^2$ and $\ell: \R \to \R$ such that $\ell \circ p_1 = p_2 \circ h$, it follows that $h$ is an equivalence homeomorphism between $\mathscr{F}(p_1)$ and $\mathscr{F}(p_2)$, proving \ref{ccc1}. 
If $Z_{p_1} = Z_{p_2} = \emptyset$ there is nothing to do. 
On the other hand, taking $\sigma= \ell|_{r_1(Z_{p_1})}$, and considering Lemma \ref{prop-ida-1} and Theorem \ref{bifurcation}, it follows that $\sigma$ satisfies \ref{ccc2}. 

(Sufficiency) We take $p_3(x,y) = p_2(x,\tau y)$ with $\tau = 1$ or $-1$ if $h$ preserves or reverses orientation, respectively. 
In any case $p_3$ is topologically equivalent to $p_2$, with $Z_{p_3} = Z_{p_2}$ and $r_3(Z_{p_3}) = r_2(Z_{p_2})$. 

If $Z_{p_1}$ is empty, it follows by Theorem \ref{mmain} that $p_1$ and $p_3$ are o-topologically equivalent. 
If $r_1(Z_{p_1})$ has at least two elements, we take $\kappa = \pm 1$ such that $\kappa \sigma: r_1(Z_{p_1}) \to \kappa r_3 (Z_{p_3})$ is increasing. 
Then again by Theorem \ref{mmain} it follows that $p_1$ and $\kappa p_3$, which is topologically equivalent to $p_3$, are o-topologically equivalent. 

So we assume that $r_1(Z_{p_1}) = \{c\}$, where $c = r_1(a) = p_1(\nu_a)$ for any $a\in Z_{p_1}$. 
We take any ordinary leaf $\xi \in \mathscr{F}(p_1)$ such that $m = p_1(\xi) > c$ and let $\kappa = \pm 1$ be such that $\kappa p_3(h(\xi)) > \kappa p_3(h(\nu_a))$. 
Then $\widetilde{\sigma}: \{c,m\} \to \{\kappa \sigma(c), \kappa p_3(h(\xi))\}$, with $\widetilde{\sigma}(c) = \kappa \sigma(c)$ and $\widetilde{\sigma}(m) = \kappa p_3(h(\xi))$, is increasing, and Theorem \ref{mmain} applies once more to conclude that $p_1$ and $\kappa p_3$ are o-topologically equivalent. 
\end{proof}

Theorem \ref{corollary:main} and Corollary \ref{varcorollary:main} are direct consequences of Theorem \ref{mmain} and Corollary \ref{cor5.1}, respectively, after applying Markus Theorem \ref{teo-markus}. 

\begin{proof}[Proof of Theorem \ref{theorem:main}]
In case $B(p_1) = B(p_2) = \emptyset$, the result follows by Theorem \ref{mmain}. 
We assume $B(p_1)\neq \emptyset$. 

For the necessity, if $\ell \circ p = q\circ h$, with $\ell:\R\to\R$ and $h: \R^2 \to \R^2$ orientation preserving homeomorphisms, then $h$ is an o-equivalence homeomorphism between $\mathscr{F}(p)$ and $\mathscr{F}(q)$, proving \ref{c1}. 
Statement \ref{c2} follows trivially by defining $\sigma = \ell|_{B(p)}$ (or $\overline{\sigma}=\ell|_{B(p)\cup\{m\}}$ for any $m\notin B(p_1)$ in case $B(p_1)$ is a singleton) and considering Lemma \ref{prop-ida-1}. 

Now the sufficiency will follow from Theorem \ref{mmain} by showing that assumption \ref{c2} of Theorem \ref{theorem:main} implies assumption \ref{cc2} of Theorem \ref{mmain} with $p_1=p$ and $p_2=q$. 
Indeed, for each $a\in Z_{p_1}$, the leaf $h(\nu_a)$ is a separatrix of $\mathscr{F}(p_2)$ and so by Proposition \ref{prop-principal} and Theorem \ref{bifurcation} it follows that $q(h(\nu_a)) =\{c\}$ is contained in $B(q)$ (actually, as in the proof of Theorem \ref{mmain}, we have that $h(\nu_a)) = \nu_b$ for certain $b\in Z_{p_2}$). 
Then from assumption \ref{c2} of Theorem \ref{theorem:main} it follows that $h^{-1}(q^{-1}(c)) = p^{-1}(\sigma^{-1}(c))$, hence $\nu_a \subset p^{-1}(\sigma^{-1}(c))$, and so $\sigma(p(\nu_a)) = \{c\}$. 
This finishes the proof in case $B(p_1)$ is not a singleton. 
If $B(p_1)$ has only one element, any $\xi \subset p_1^{-1}(m)$, with $m$ given in \ref{c2} of Theorem \ref{theorem:main} will apply to prove \ref{cc2} of Theorem \ref{mmain}. 
\end{proof}
 
\begin{proof}[Proof of Corollary \ref{coro:main}]
It follows similarly as the proof of Corollary \ref{cor5.1} above. 
\end{proof} 

We think that results like theorems \ref{theorem:main} and \ref{corollary:main} as well as their corollaries are true in a wider class of submersion functions, namely the ones where we do not have the vanishing at infinity phenomenon, that already does not appear in our class of functions, according to Corollary \ref{12091}. 
Anyway, our proofs need special regions of the foliation $\mathscr{F}(p)$ where $p$ assumes all the values in $\R$, as the reasons in this section make clear. 
New techniques must be developed in order to control this in the general case. 

\section{Acknowledgements}
We thank Professor Luis Renato Gon\c{c}alves Dias for reading and commenting on a previous version of this article. 
The first named author is partially supported by the grants 2019/07316-0 and 2020/14498-4, S\~ao Paulo Research Foundation (FAPESP). 
The second named author is partially supported by Coordena\c{c}\~ao de Aperfei\c{c}oamento de Pessoal de N\'ivel Superior - Brasil (CAPES) - Finance Code 001.


\end{document}